\documentclass{hha}
\usepackage[alphabetic]{amsrefs}
\usepackage[utf8]{inputenc}
\usepackage{amsmath,amssymb,amsthm}
\usepackage{url}

\usepackage[shortlabels]{enumitem}

\usepackage[colorlinks=true,linktocpage=true,citecolor=blue]{hyperref}

\usepackage{mbboard}

\usepackage{eucal}

\newcommand{\blank}{\text{\textendash}}
\newcommand{\defterm}[1]{\emph{#1}}
\newcommand{\isoto}{\xrightarrow{\sim}}

\newcommand{\IFF}{if and only if}

\newcommand{\catname}[1]{\ensuremath{\text{\textup{#1}}}}
\newcommand{\txt}[1]{\ensuremath{\text{\textup{#1}}}}
\newcommand{\Set}{\catname{Set}}
\newcommand{\sSet}{\Set_{\Delta}}
\newcommand{\Cat}{\catname{Cat}}
\newcommand{\CatI}{\catname{Cat}_\infty}
\newcommand{\LCatI}{\widehat{\catname{Cat}}_\infty}

\newcommand{\Fun}{\txt{Fun}}

\newcommand{\Map}{\txt{Map}}

\newcommand{\op}{\txt{op}}

\newcommand{\icat}{$\infty$-category}
\newcommand{\icats}{$\infty$-categories}
\newcommand{\icatl}{$\infty$-categorical}

\newcommand{\xto}[1]{\xrightarrow{#1}}

\usepackage{tikz}
\usetikzlibrary{matrix,arrows}

\newcommand{\csquare}[8]{ %
\[ %
\begin{tikzpicture} %
\matrix (m) [matrix of math nodes,row sep=3em,column sep=2.5em,text height=1.5ex,text depth=0.25ex] %
{ #1 \pgfmatrixnextcell #2 \\ %
  #3 \pgfmatrixnextcell #4 \\ }; %
\path[->,font=\footnotesize] %
(m-1-1) edge node[auto] {$#5$} (m-1-2)%
(m-1-1) edge node[left] {$#6$} (m-2-1)%
(m-1-2) edge node[auto] {$#7$} (m-2-2)%
(m-2-1) edge node[below] {$#8$} (m-2-2);%
\end{tikzpicture}%
\]%
}

\newcommand{\nolabelcsquare}[4]{\csquare{#1}{#2}{#3}{#4}{}{}{}{}}

\newcommand{\opctriangle}[6]{ %
\[ %
\begin{tikzpicture} %
\matrix (m) [matrix of math nodes,row sep=3em,column sep=1.2em,text height=1.5ex,text depth=0.25ex] %
{  #1 \pgfmatrixnextcell \pgfmatrixnextcell #2 \\ %
  \pgfmatrixnextcell #3 \pgfmatrixnextcell \\ %
}; %
\path[->,font=\footnotesize] %
(m-1-1) edge node[above] {$#4$} (m-1-3)%
(m-1-1) edge node[below left] {$#5$} (m-2-2)%
(m-1-3) edge node[below right] {$#6$} (m-2-2);%
\end{tikzpicture}%
\]%
}

\newcommand{\nolabelopctriangle}[3]{\opctriangle{#1}{#2}{#3}{}{}{}}

\newcommand{\id}{\txt{id}}

\DeclareMathOperator{\colimP}{colim}
\newcommand{\colim}{\mathop{\colimP}}

\newcommand{\CatIV}{\CatI^{\mathcal{V}}}
\newcommand{\CATIV}{\txt{CAT}_{\infty}^{\mathcal{V}}}

\newcommand{\FunV}{\Fun^{\mathcal{V}}}

\newcommand{\simp}{\bbDelta}

\newcommand{\Seg}{\txt{Seg}}

\newcommand{\SegI}{\Seg_{\infty}}

\newcommand{\Alg}{\catname{Alg}}

\newcommand{\AlgCat}{\Alg_{\txt{cat}}}

\newcommand{\AlgCatV}{\AlgCat(\mathcal{V})}

\newcommand{\Opd}{\catname{Opd}}
\newcommand{\OpdI}{\Opd_{\infty}}

\newcommand{\OpdIns}{\OpdI^{\txt{ns}}}
\newcommand{\OpdInsg}{\OpdI^{\txt{ns},\txt{gen}}}

\newcommand{\iopd}{$\infty$-operad}
\newcommand{\iopds}{$\infty$-operads}

\newcommand{\nsiopd}{non-symmetric $\infty$-operad}
\newcommand{\nsiopds}{non-symmetric $\infty$-operads}

\newcommand{\gnsiopd}{generalized non-symmetric $\infty$-operad}
\newcommand{\gnsiopds}{generalized non-symmetric $\infty$-operads}

\theoremstyle{definition}
\newtheorem{construction}[theorem]{Construction}

\usepackage{tikz-cd}

\newcommand{\Dnop}{\simp^{n,\op}}

\newcommand{\LOpdI}{\widehat{\txt{Opd}}_{\infty}}

\newcommand{\Lbrn}{\bbLambda_{/[n]}}
\newcommand{\Lbrnop}{\Lbrn^{\op}}

\newcommand{\ALG}{\txt{ALG}}

\newcommand{\Gop}{\bbGamma^{\op}}

\newcommand{\fALG}{\mathfrak{ALG}}

\newcommand{\ofALG}{\overline{\fALG}}

\newcommand{\Dop}{\simp^{\op}}
\newcommand{\Lop}{\bbLambda^{\op}}

\newcommand{\Dopn}{\simp^{\op}_{/[n]}}

\renewcommand{\OpdIns}{\OpdI^{\txt{ns}}}
\renewcommand{\OpdInsg}{\OpdI^{\txt{ns},\txt{gen}}}
\newcommand{\Bimod}{\txt{Bimod}}

\newcommand{\ALGcat}{\ALG_{\txt{cat}}}
\newcommand{\fALGcat}{\fALG_{\txt{cat}}}
\newcommand{\ofALGcat}{\ofALG_{\txt{cat}}}
\newcommand{\fCAT}{\mathfrak{CAT}}

\newcommand{\Lopn}{\Lop_{/[n]}}
\newcommand{\DopXn}{\Dop_{X_{0},\ldots,X_{n}}}

\newcommand{\DopXYn}{\Dop_{X_{0}\times Y_{0},\ldots,X_{n} \times Y_{n}}}
\newcommand{\LopXn}{\Lop_{X_{0},\ldots,X_{n}}}

\newcommand{\LopXYn}{\Lop_{X_{0}\times Y_{0},\ldots,X_{n} \times Y_{n}}}

\renewcommand{\FunV}{\Fun_{\mathcal{V}}}

\newcommand{\DopX}{\Dop_{X}}

\newcommand{\BMcat}{\Bimod_{\txt{cat}}}
\newcommand{\Algcat}{\Alg_{\txt{cat}}}
\newcommand{\LOpdInsg}{\LOpdI^{\txt{ns},\txt{gen}}}

\begin{document}

\title{Bimodules and Natural Transformations for Enriched $\infty$-Categories}


\author{Rune Haugseng}             

\email{haugseng@mpim-bonn.mpg.de}

%
%
\address{Max-Planck-Institut für Mathematik,
  Vivatsgasse 7,
  53111 Bonn,
  Germany}




\begin{abstract}
  We introduce a notion of bimodule in the setting of enriched
  $\infty$-categories, and use this to construct a double
  $\infty$-category of enriched $\infty$-categories where the two
  kinds of 1-morphisms are functors and bimodules. We then consider a
  natural definition of natural transformations in this context, and
  show that in the underlying $(\infty,2)$-category of enriched
  $\infty$-categories with functors as 1-morphisms the 2-morphisms are
  given by natural transformations.
\end{abstract}

\received{\ldots}   
\revised{\ldots}    
\published{\ldots}  
\submitted{}      
\volumeyear{0} 
\volumenumber{0} 
\issuenumber{0}   
\startpage{0}     
\articlenumber{0} 

\maketitle

\section{Introduction}
This paper is a sequel to \cite{enr} and part of \cite{nmorita}: In
\cite{enr}, David Gepner and I set up a theory of \emph{enriched
  \icats{}}, using a non-symmetric variant of Lurie's theory of
\iopds{}, and in \cite{nmorita}*{\S 5} I constructed a double \icat{} 
$\fALG(\mathcal{V})$ of associative algebra objects in a monoidal
\icat{} $\mathcal{V}$, with the two kinds of 1-morphism given by
algebra homomorphisms and bimodules. The goal of this
paper is to construct a ``many-object'' analogue of this double
\icat{}: In \cite{enr} we defined enriched \icats{} as algebras for
``many-object associative operads'', and there is an analogous
extension of the definition of bimodules in \cite{nmorita} using
``many-object bimodule operads''. Using this definition we extend the
constructions of \cite{nmorita} to get our main result:
\begin{theorem}
  Let $\mathcal{V}$ be a monoidal \icat{} compatible with small
  colimits. Then there exists a double \icat{} $\fALGcat(\mathcal{V})$
  of $\mathcal{V}$-enriched \icats{}, with the two kinds of 1-morphism
  given by bimodules and functors.  Moreover, if $\mathcal{V}$ is an
  $\mathbb{E}_{n+1}$-monoidal \icat{}, then $\fALGcat(\mathcal{V})$
  inherits a natural $\mathbb{E}_{n}$-monoidal structure.
\end{theorem}
We'll construct this double \icat{} in \S\ref{sec:double} and discuss
its functoriality and monoidal structures in \S\ref{sec:func}.

We can also restrict the objects of this double \icat{} to those
$\mathcal{V}$-\icats{} that are \emph{complete}, i.e. local with
respect to the fully faithful and essentially surjective functors, to
obtain the double \icat{} $\fCAT(\mathcal{V})$, which we regard as the
``correct'' double \icat{} of
$\mathcal{V}$-\icats{}.\footnote{However, we do not show here that
  this double subcategory is functorial or inherits the monoidal
  structures on $\fALGcat(\mathcal{V})$ --- this is a consequence of
  the Yoneda Lemma, which we hope to prove in a sequel to this paper.}

The double \icat{} $\fALGcat(\mathcal{V})$ has two underlying
$(\infty,2)$-categories, with the 1-morphisms given either by
bimodules or by functors. In the latter case, we would expect the
2-morphisms to be \emph{natural transformations}. The second main
result of this paper is that this is indeed the case: We will use the
obvious notion of a natural transformation of functors between
$\mathcal{V}$-\icats{} $\mathcal{C}$ and $\mathcal{D}$, namely a
functor \[\mathcal{C} \otimes [1] \to \mathcal{D},\] to define a Segal
space $\FunV(\mathcal{C}, \mathcal{D})$ of $\mathcal{V}$-functors, and
show:
\begin{theorem}
  Let $\ALGcat(\mathcal{V})$ be the $(\infty,2)$-category (in the
  sense of a 2-fold Segal space) underlying $\fALGcat(\mathcal{V})$
  with functors as 1-morphisms. There is a natural equivalence between
  $\FunV(\mathcal{C}, \mathcal{D})$ and the Segal space
  $\ALGcat(\mathcal{V})(\mathcal{C}, \mathcal{D})$ of maps from
  $\mathcal{C}$ to $\mathcal{D}$ in $\fALGcat(\mathcal{V})$.
\end{theorem}
We'll prove this in \S\ref{sec:infty2}. If $\mathcal{D}$ is complete
we will also observe that the Segal space $\FunV(\mathcal{C}, \mathcal{D})$
is complete for any $\mathcal{C}$, so as a consequence we obtain that
the 2-fold Segal space $\CATIV$ underlying $\fCAT(\mathcal{V})$ with
functors as 1-morphisms is complete.

In ordinary enriched category theory the notion of bimodule is
classical, and according to the $n$lab was invented independently by a
number of people back in the 1960s, though with much of their theory
introduced by B\'{e}nabou. The specific definition of a bimodule between
enriched $\infty$-categories we consider here was, however, inspired by the
``external'' notion of bimodule given by Bacard in
\cite{BacardSegEnrI} in the context of a model-categorical approach to
weakly enriched categories.

To motivate this paper, let's now briefly consider some future
directions in which I hope to extend the results proved here:
\begin{itemize}
\item In \cite{nmorita}*{\S 6} I constructed for any
  $\mathbb{E}_{n}$-monoidal \icat{} an
  $(\infty,n+1)$-category of $\mathbb{E}_{n}$-algebras in it. Similarly, I hope to construct
  $(\infty,n+1)$-categories of enriched $(\infty,n)$-categories also
  for $n > 1$ --- these are expected to be the targets for a number of
  interesting extended topological quantum field theories.
\item In \cite{HA}*{\S 4.6.3} Lurie proves that all associative
  algebras are dualizable in the \icat{} of algebras and
  bimodules. This should extend to a proof that all enriched \icats{}
  are dualizable, which will lead to a definition of \emph{topological
    Hochschild homology} for enriched \icats{}. Similarly, the proof
  in \cite{HA}*{\S 4.6.4} that the 2-dualizable algebras are precisely
  the smooth and proper ones should extend to a characterization of
  the 2-dualizable enriched \icats{}.
\item For ordinary enriched categories, a bimodule between
  $\mathbf{V}$-categories $\mathbf{C}$ and $\mathbf{D}$ is often
  defined as a functor from $\mathbf{C} \otimes \mathbf{D}^{\op}$ to
  the self-enrichment of $\mathbf{V}$. The same should be true for the
  bimodules we consider here: the \icat{} of
  $\mathcal{C}$-$\mathcal{D}$-bimodules in $\mathcal{V}$ should be a
  representable functor of $\mathcal{C}$, with the representing object
  being $\mathcal{V}$-valued enriched presheaves on
  $\mathcal{D}$. This can be thought of as a form of the \emph{Yoneda
    Lemma} for enriched \icats{}. (In particular, the more obvious
  formulation that there is a fully faithful Yoneda embedding into
  enriched presheaves would be an easy consequence of this.)
\item Classically, the double category of $\mathbf{V}$-enriched
  categories, functors, and bimodules is an example of a
  \emph{proarrow equipment}. This is an abstract context in which one
  can define weighted (co)limits and Kan extensions. An analogous
  theory can be developed in the \icatl{} context, with the double
  \icat{} we construct here as a key example. Combined with the Yoneda
  Lemma, which gives a checkable criterion for a bimodule to be
  represented by a functor, this should give very useful tools for
  making interesting constructions and in general ``doing category
  theory'' with enriched \icats{} (with a particularly interesting
  case here being $(\infty,n)$-categories).
\end{itemize}

\subsection{Overview}
In \S\ref{sec:nsop} we review some key notions and results from the
theory of non-symmetric \iopds{}, and in \S\ref{sec:enr} we briefly
recall the main definitions and results on enriched \icats{} from
\cite{enr} that we'll make use of. Then in \S\ref{sec:bimod} we
introduce our definition of bimodules between enriched \icats{}, and
motivate it by relating it to the classical notion of a bimodule for
enriched categories. Next we discuss, in \S \ref{sec:compose}, how to
compose these bimodules, and observe that this is analogous to the
composition of bimodules for ordinary enriched categories. After these
introductory sections we then get to work in \S\ref{sec:double}, where
we construct the double \icat{} of enriched \icats{}. In
\S\ref{sec:nattr} we consider the obvious definition of natural
transformations in this context and show these are the 1-morphisms in
an \icat{} of enriched functors, and then we compare this to the
mapping \icat{} of functors coming from our double \icat{} in
\S\ref{sec:infty2}. Finally we discuss the functoriality of the double
\icats{} and their natural monoidal structures in \S\ref{sec:func}.

\subsection{Notation}
We recycle the notation of \cite{enr} and \cite{nmorita}. In
particular, for $[n]$ an object of $\simp$ we'll abbreviate
$(\simp_{/[n]})^{\op}$ to $\Dopn$ to avoid clutter as this object will
appear frequently, often in subscripts. If $\phi \colon [m] \to [n]$
is an object of $\simp_{/[n]}$ we'll also denote this object by the
list $(\phi(0),\ldots, \phi(m))$ where $0 \leq \phi(i) \leq \phi(i+1)
\leq m$.

\subsection{Acknowledgments}
This is the final paper based on part of my Ph.D. thesis --- though
much improved by being left to stew for a while --- so it is a
pleasure to get to thank Haynes Miller once more for being a great
advisor, as well as the Norway-America Association and the
American-Scandinavian Foundation for partially funding my studies at
MIT. This also seems an appropiate occasion to thank David Gepner for
steering me away from a truly atrocious approach to defining \icats{} of
functors between enriched \icats{} back in 2012.

\section{Non-Symmetric $\infty$-Operads}\label{sec:nsop}
Here we briefly recall some of the basic definitions from the theory of
(non-symmetric) \iopds{} and summarize some key results that we
will use in this paper. For motivation for these definitions we refer
the reader to the discussion in \cite{enr}*{\S 2}, and for
proofs we refer to \cite{enr}*{\S 3 and \S A}, and of course \cite{HA}.

\begin{definition}\label{defn:activeinert}
  Let $\simp$ be the usual simplicial indexing category. A morphism $f
  \colon [n] \to [m]$ in $\simp$ is \defterm{inert} if it is the
  inclusion of a sub-interval of $[m]$, i.e. $f(i) = f(0)+i$ for all
  $i$, and \defterm{active} if it preserves the extremal elements,
  i.e.  $f(0) = 0$ and $f(n) = m$. We say a morphism in $\simp^{\op}$
  is \emph{active} or \emph{inert} if it is so when considered as a
  morphism in $\simp$, and write $\simp^{\op}_{\txt{act}}$ and
  $\simp^{\op}_{\txt{int}}$ for the subcategories of $\simp^{\op}$
  with active and inert morphisms, respectively. We write $\rho_{i}
  \colon [n] \to [1]$ for the inert map in $\simp^{\op}$ corresponding
  to the inclusion $\{i-1,i\} \hookrightarrow [n]$.
\end{definition}

\begin{definition}\label{defn:gnsiopd}
  A \defterm{generalized non-symmetric $\infty$-operad}
  is an inner fibration $\pi \colon \mathcal{M} \to
  \simp^{\op}$ such that:
  \begin{enumerate}[(i)]
  \item For each inert map  $\phi \colon [n] \to [m]$ in $\simp^{\op}$ and
    every $X \in \mathcal{M}$ such that $\pi(X) = [n]$, there exists
    a $\pi$-coCartesian edge $X \to \phi_{!}X$ over $\phi$.
  \item For every $[n]$ in $\simp^{\op}$, the map
    \[ \mathcal{M}_{[n]} \to \mathcal{M}_{[1]}
    \times_{\mathcal{M}_{[0]}} \cdots \times_{\mathcal{M}_{[0]}} \mathcal{M}_{[1]}
    \]
    induced by the inert maps $[n] \to [1],[0]$ is an equivalence.
  \item Given $C \in \mathcal{M}_{[n]}$ and a coCartesian map $C \to
    C_{\alpha}$ over each inert map $\alpha$ from $[n]$ to
    $[1]$ and $[0]$, the object $C$ is a $\pi$-limit of the
    $C_{\alpha}$'s.
  \end{enumerate}
  A \emph{non-symmetric \iopd{}} is a \gnsiopd{} $\mathcal{M}$ such
  that $\mathcal{M}_{0} \simeq *$.
\end{definition}

\begin{definition}
  A \emph{double \icat{}} is a \gnsiopd{} $\mathcal{M} \to \Dop$ that
  is also a coCartesian fibration, and a \emph{monoidal \icat{}} is a
  \nsiopd{} that is also a coCartesian fibration.
\end{definition}
Equivalently, a double \icat{} can be defined as a coCartesian
fibration such that the associated functor $F \colon \Dop \to \CatI$
satisfies the \emph{Segal condition}: for every $[n] \in \Dop$, the
functor
\[ F([n]) \to F([1]) \times_{F([0])} \cdots \times_{F([0])} F([1]), \]
induced by the maps $\rho_{i}\colon [n] \to [1]$ and all the maps
$[n] \to [0]$, is an equivalence of \icats{}.

\begin{definition}
  A morphism of (generalized) \nsiopds{} is a commutative diagram
  \opctriangle{\mathcal{M}}{\mathcal{N}}{\simp^{\op}}{\phi}{}{} such
  that $\phi$ carries inert morphisms in $\mathcal{M}$ to inert
  morphisms in $\mathcal{N}$. We will also refer to a morphism of
  (generalized) \nsiopds{} $\mathcal{M} \to \mathcal{N}$ as an
  \emph{$\mathcal{M}$-algebra} in $\mathcal{N}$; we write
  $\Alg_{\mathcal{M}}(\mathcal{N})$ for the full subcategory of the
  \icat{} $\Fun_{\simp^{\op}}(\mathcal{M},\mathcal{N})$ of functors
  over $\simp^{\op}$ spanned by the morphisms of (generalized)
  \nsiopds{}.
\end{definition}

Using the theory of categorical patterns developed in \cite{HA}*{\S B}, we can define \icats{}
$\OpdIns$ and $\OpdInsg$ of \nsiopds{} and \gnsiopds{}. The \icats{}
of algebras are functorial in these \icats{}, and indeed determine a
lax monoidal functor
\[ (\OpdInsg)^{\op} \times \OpdInsg \to \CatI.\]
For a \gnsiopd{} $\mathcal{M}$ we define the \emph{algebra
  fibration} \[\Alg(\mathcal{M}) \to \OpdInsg\] to be a Cartesian
fibration associated to the functor \[\Alg_{(\blank)}(\mathcal{M})
\colon (\OpdInsg)^{\op} \to \CatI.\]

We say that a monoidal \icat{} $\mathcal{V}^{\otimes}$ is
\emph{compatible with small colimits} if the underlying \icat{}
$\mathcal{V}$ has small colimits and the tensor product preserves
colimits in each variable. If $\mathcal{V}$ is compatible with small
colimits and $f \colon \mathcal{O} \to \mathcal{P}$ is a morphism of
small \iopds{}, then the functor $f^{*} \colon
\Alg_{\mathcal{P}}(\mathcal{V}) \to \Alg_{\mathcal{O}}(\mathcal{V})$
has a left adjoint $f_{!}$, given by taking operadic left Kan
extensions along $f$. If $A$ is an $\mathcal{O}$-algebra in
$\mathcal{V}$, the value of the $\mathcal{P}$-algebra $f_{!}A$ at
$p \in \mathcal{P}_{[1]}$ is given by a certain colimit,\footnote{In
  fact, by \cite{HA}*{Corollary 3.1.3.5} this colimit description
essentially characterizes $f_{!}A$, in the sense that a morphism $\phi
\colon A \to f^{*}B$ is adjunct to an equivalence $f_{!}A \isoto B$
\IFF{} $\phi$ induces an equivalence between the appropriate colimit
and $B(p)$ for all $p \in \mathcal{P}_{[1]}$.} which we can somewhat informally express as
\[ \colim_{(o_{1},\ldots,o_{n}) \in \mathcal{O}^{\txt{act}}_{/p}}
A(o_{1}) \otimes \cdots \otimes A(o_{n}).\]
Here $\mathcal{O}^{\txt{act}}_{/p}$ is the \icat{}
$\mathcal{O}^{\txt{act}} \times_{\mathcal{P}^{\txt{act}}}
\mathcal{P}^{\txt{act}}_{/p}$ of objects of $\mathcal{O}$ whose image
in $\mathcal{P}$ has an active map to $p$, and active maps between
them.

In good cases we can also explicitly describe this left adjoint in the
same way when $f$ is a morphism of \gnsiopds{} (by
\cite{nmorita}*{Theorem 4.40}), namely if $f$ is
\emph{extendable} in the sense of \cite{nmorita}*{Definition 4.38},
which we recall here for the reader's convenience:
\begin{definition}\label{defn:extendable}
  A morphism $f \colon \mathcal{O} \to \mathcal{P}$ of \gnsiopds{} is
  \emph{extendable} if for every $P \in \mathcal{P}$ lying over $[n]
  \in \Dop$, the morphism
  \[ \mathcal{O}^{\txt{act}}_{/P} \to \prod_{i = 1}^{n}
  \mathcal{O}^{\txt{act}}_{/P_{i}},\] induced by the coCartesian
  morphisms in $\mathcal{O}$ over the maps $\rho_{i}$, is cofinal.
\end{definition}

 \section{Enriched $\infty$-Categories}\label{sec:enr}
In this section we recall the definition of enriched \icats{} as
``many-object associative algebras'' we introduced in \cite{enr}, and
some key definitions and results from that paper that we will make use
of here. For further motivation for this definition we refer to
\cite[\S 2]{enr}, and for complete details of the constructions we
refer to \cite[\S 4--5]{enr}.

\begin{definition}
  Given a space $X$, let $i_{X} \colon \Dop \to \mathcal{S}$ be the
  right Kan extension of the functor $\{[0]\} \to \mathcal{S}$ that sends $[0]$
  to $X$. We write $\Dop_{X} \to \Dop$ for the left fibration
  associated to the functor $i_{X}$.
\end{definition}

\begin{remark}
  The functor $i_{X}$ sends $[n]$ to $X^{\times (n+1)}$ and takes face
  maps to projections and degeneracies to diagonal maps. We can thus
  identify the objects of $\Dop_{X}$ lying over $[n]$ with lists
  $(x_{0},\ldots,x_{n})$ of points $x_{i} \in X$.
\end{remark}

\begin{lemma}[\cite{enr}*{Lemma 4.1.3}]
  For any space $X$, the projection $\Dop_{X}\to \Dop$ is a
  double \icat{}.
\end{lemma}

\begin{definition}
  If $\mathcal{V}$ is a monoidal \icat{}, a
  \emph{$\mathcal{V}$-enriched \icat{}} (or
  \emph{$\mathcal{V}$-\icat}) with space of objects $X$ is a
  $\Dop_{X}$-algebra in $\mathcal{V}$.
\end{definition}

\begin{remark}
  Observe that an algebra $\mathcal{C} \colon \Dop_{X} \to
  \mathcal{V}^{\otimes}$ assigns to every $(x,y) \in X^{\times 2}$ an
  object $\mathcal{C}(x,y) \in \mathcal{V}$. Speaking slightly
  informally, for every $(x,y,z) \in X^{\times 3}$ the face map
  $(x,y,z) \to (x,z)$ in $\Dop_{X}$ gives a composition map
  $\mathcal{C}(x,y) \otimes \mathcal{C}(y,z) \to \mathcal{C}(x,z)$,
  and the degeneracy $(x) \to (x,x)$ gives a unit $I \to
  \mathcal{C}(x,x)$ where $I$ is the unit of the monoidal \icat{}
  $\mathcal{V}$. The remaining data in $\mathcal{C}$ shows that this
  composition is coherently homotopy-associative and compatible with
  the unit --- this is precisely the data we would expect an enriched
  \icat{} to be given by.
\end{remark}

\begin{definition}
  The \icats{} $\Dop_{X}$ are clearly functorial in $X$, and so
  determine a functor $\mathcal{S} \to \OpdInsg$. If $\mathcal{V}$ is
  a monoidal \icat{}, we let $\Algcat(\mathcal{V}) \to \mathcal{S}$ be
  a Cartesian fibration associated to the functor $\mathcal{S}^{\op}
  \to \CatI$ that sends $X$ to $\Alg_{\Dop_{X}}(\mathcal{V})$.
\end{definition}

If we take $\mathcal{V}$ to be the \icat{} $\mathcal{S}$ of spaces,
then $\Algcat(\mathcal{S})$ is equivalent to the \icat{} $\SegI$ of
Segal spaces.

The \icat{} $\Algcat(\mathcal{V})$ is functorial in $\mathcal{V}$, and
it is a lax monoidal functor with respect to the Cartesian
product of monoidal \icats{} and of \icats{}. Moreover, if
$\mathcal{V}$ is compatible with small colimits then
$\Algcat(\mathcal{V})$ has small colimits, and it is tensored over
$\Algcat(\mathcal{S})$ in such a way that the tensoring preserves
colimits in each variable. Any Segal space thus gives rise to a
$\mathcal{V}$-\icat{} by tensoring with the unit of $\mathcal{V}$,
regarded as a 1-object $\mathcal{V}$-\icat{}.

In particular, if we write $E^{n}$ for the contractible category with
objects $\{0,\ldots,n\}$ and a unique morphism $i \to j$ for all $i$
and $j$, this determines a $\mathcal{V}$-\icat{} we also denote
$E^{n}$.\footnote{In fact, we can define $E^{n}$ as a
  $\mathcal{V}$-\icat{} for an arbitrary $\mathcal{V}$, but we will
  not need this generality here.}
We say that a $\mathcal{V}$-\icat{} $\mathcal{C}$ is \emph{complete}
if it is local with respect to the map $E^{1} \to E^{0}$, i.e. if the
map of spaces $\Map(E^{0}, \mathcal{C}) \to \Map(E^{1}, \mathcal{C})$
is an equivalence. Under the equivalence between
$\Algcat(\mathcal{S})$ and Segal spaces, the complete
$\mathcal{S}$-\icats{} precisely correspond to the \emph{complete}
Segal spaces in the sense of Rezk.

We write $\CatIV$ for the full subcategory of $\Algcat(\mathcal{V})$
spanned by the complete $\mathcal{V}$-\icats{}. Our main result in
\cite{enr} was that the inclusion $\CatIV \hookrightarrow
\Algcat(\mathcal{V})$ has a left adjoint, which exhibits $\CatIV$ as
the localization of $\Algcat(\mathcal{V})$ with respect to the class
of \emph{fully faithful and essentially surjective}
morphisms\footnote{We do not recall the definition of these here, as
  we will not make use of this class of morphisms in this
  paper.}. This means that $\CatIV$ is the ``correct'' homotopy theory
of $\mathcal{V}$-enriched \icats{}.

\section{Bimodules}\label{sec:bimod}
If $\mathbf{V}$ is a closed symmetric monoidal category, so that there
is a tensor product of $\mathbf{V}$-categories and $\mathbf{V}$ has a
natural self-enrichment $\overline{\mathbf{V}}$, the classical
definition of a \emph{bimodule} between $\mathbf{V}$-categories
$\mathbf{C}$ and $\mathbf{D}$ is a $\mathbf{V}$-functor
\[ M \colon \mathbf{C}^{\op} \otimes \mathbf{D} \to
\overline{\mathbf{V}}.\] 
We can reformulate this definition to see it as a many-object
version of the usual notion of a bimodule for associative algebras:
unravelling the definition, a $\mathbf{C}$-$\mathbf{D}$-bimodule
consists of:
\begin{itemize}
\item for all $c \in \mathbf{C}$ and $d \in \mathbf{D}$, an object
  $M(c,d) \in \mathbf{V}$,
\item for all $c',c \in \mathbf{C}$ and $d \in \mathbf{D}$, an action
  map $\mathbf{C}(c',c) \otimes M(c,d) \to M(c',d)$, compatible with
  composition and units in $\mathbf{C}$,
\item for all $c \in \mathbf{C}$ and $d, d' \in \mathbf{D}$, an action
  map $M(c,d) \otimes \mathbf{D}(d,d') \to M(c,d')$, compatible with
  composition and units in $\mathbf{D}$,  
\end{itemize}
such that for $c,c' \in \mathbf{C}$ and $d,d' \in \mathbf{D}$, the diagram
\nolabelcsquare{\mathbf{C}(c',c) \otimes M(c,d) \otimes
  \mathbf{D}(d,d')}{M(c',d)\otimes \mathbf{D}(d,d')}{\mathbf{C}(c',c)
  \otimes M(c,d')}{M(c',d')}
commutes. Notice that this definition does not require $\mathbf{V}$ to
be closed or symmetric monoidal.

Since we defined enriched \icats{} as algebras for ``many-object
associative \iopds{}'', this suggests that we can define bimodules for
enriched \icats{} as algebras for ``many-object bimodule
\iopds{}''. In \cite{nmorita} we observed that bimodules can be
regarded as algebras for the \gnsiopd{} $\Dop_{/[1]} \to \Dop$. Here
is the obvious ``many-object'' version of this:
\begin{definition}
  Given spaces $X$ and $Y$, we let $\Dop_{X,Y}\to \Dop_{/[1]}$ be a
  left fibration associated to the functor $\Dop_{/[1]} \to
  \mathcal{S}$ obtained as a right Kan extension of the functor
  $\{(0), (1)\} \to \mathcal{S}$ sending $(0)$ to $X$ and $(1)$ to $Y$.
\end{definition}
The composite functor $\Dop_{X,Y} \to \Dop_{/[1]} \to \Dop$ is a
double \icat{} --- this is a special case of
Lemma~\ref{lem:DopXndouble}, which we'll prove below.

\begin{remark}
  An object of $\Dop_{/[1]}$ can be described as a list
  $(i_{0},\ldots,i_{n})$ where $0 \leq i_{k}\leq i_{k+1} \leq 1$. An
  object of $\Dop_{X_{0},X_{1}}$ lying over this is then a list
  $(x_{0},\ldots,x_{n})$ with $x_{k}\in X_{i_{k}}$. There are
  inclusions $\Dop_{X_{i}} \hookrightarrow \Dop_{X_{0},X_{1}}$ lying over the two
  inclusions $\Dop \to \Dop_{/[1]}$ (given by composing with the two
  maps $[0] \to [1]$ in $\simp$). Suppose $\mathcal{V}$ is a monoidal
  \icat{} and $M \colon \Dop_{X_{0},X_{1}} \to \mathcal{V}^{\otimes}$
  is a $\Dop_{X_{0},X_{1}}$-algebra in $\mathcal{V}$. If we write
  $\mathcal{C}$ and $\mathcal{D}$ for the two enriched \icats{}
  obtained by restricting $M$ to $\Dop_{X_{0}}$ and $\Dop_{X_{1}}$,
  the additional data determined by $M$ can be described as:
  \begin{itemize}
  \item for $c \in X_{0}$ and $d \in X_{1}$, an object
    $M(c,d) \in \mathcal{V}$,
  \item for $c',c \in X_{0}$ and $d \in X_{1}$, a morphism
    $\mathcal{C}(c',c) \otimes M(c,d) \to M(c',d)$, coming from the
    map $(c',c,d) \to (c',d)$ (over $d_{1} \colon (0,0,1) \to (0,1)$),
  \item for $c \in X_{0}$ and $d,d' \in X_{1}$, a morphism
    $M(c,d) \otimes \mathcal{D}(d,d') \to M(c,d')$, coming from the
    map $(c,d,d') \to (c,d')$ (over $d_{1} \colon (0,1,1) \to (0,1)$),
  \item for $c',c \in X_{0}$ and $d,d' \in X_{1}$, a homotopy-commutative square
    \nolabelcsquare{\mathcal{C}(c',c) \otimes M(c,d) \otimes
      \mathcal{D}(d,d')}{M(c,d) \otimes
      \mathcal{D}(d,d')}{\mathcal{C}(c',c) \otimes M(c,d)}{M(c',d'),}
    since the two maps $(c',c,d,d') \to (c',d')$ are homotopic,
  \item together with data showing that these action maps are
    homotopy-coherently compatible with the composition and unit maps
    in $\mathcal{C}$ and $\mathcal{D}$.
  \end{itemize}
  In other words, $M$ is precisely a homotopy-coherent version of the
  notion of bimodule for enriched categories we considered above.
\end{remark}

\begin{definition}
  The \gnsiopds{} $\Dop_{X,Y}$ are clearly natural in $X$ and $Y$, so
  we get a functor $\mathcal{S}^{\times 2} \to \OpdInsg$. If
  $\mathcal{V}$ is a monoidal \icat{}, we let $\BMcat(\mathcal{V}) \to
  \mathcal{S}^{\times 2}$ be a Cartesian fibration associated to the
  functor sending $(X,Y)$ to $\Alg_{\Dop_{X,Y}}(\mathcal{V})$. There
  are natural maps of \gnsiopds{} $\Dop_{X}, \Dop_{Y} \hookrightarrow
  \Dop_{X,Y}$, which leads to a functor $\BMcat(\mathcal{V}) \to
  \Algcat(\mathcal{V})^{\times 2}$. If $\mathcal{C}$ and $\mathcal{D}$
  are $\mathcal{V}$-\icats{}, we call an object of the fibre of this
  map at $(\mathcal{C},\mathcal{D})$ a
  \emph{$\mathcal{C}$-$\mathcal{D}$-bimodule}.
\end{definition}

\section{Composing Bimodules}\label{sec:compose}
Let $\mathbf{V}$ be an ordinary monoidal category. If $\mathbf{A}$,
$\mathbf{B}$, and $\mathbf{C}$ are $\mathbf{V}$-categories and we are
given an $\mathbf{A}$-$\mathbf{B}$-bimodule $M$ and a
$\mathbf{B}$-$\mathbf{C}$-bimodule $N$, their \emph{composite}, which
we'll denote $M \otimes_{\mathbf{B}} N$, is the
$\mathbf{A}$-$\mathbf{C}$-bimodule given by sending $(a,c)$ to the
coequalizer
\[ \coprod_{b,b' \in \mathbf{B}} M(a,b) \otimes \mathbf{B}(b,b')
\otimes N(b',c) \rightrightarrows \coprod_{b \in \mathbf{B}} M(a,b) \otimes
N(b,c),\] 
with the two maps given by the action of $\mathbf{B}$ on $M$
and $N$. In fact, this is a reflexive coequalizer, since we get a map
in the other direction using the unit maps of $\mathbf{B}$. When
passing from ordinary categories to \icats{} the natural replacement
of a reflexive coequalizer is usually the geometric realization of a
simplicial object, and indeed there is a natural simplicial object
extending this coequalizer diagram, namely:
  \[ 
\begin{tikzpicture} %
\matrix (m) [matrix of math nodes,row sep=1.5em,column sep=1.5em,text
height=1.5ex,text depth=0.25ex]
{ \vdots \\
\displaystyle{\coprod_{b,b',b'',b''' \in \mathbf{B}}} M(a,b) \otimes
\mathbf{B}(b,b') \otimes \mathbf{B}(b',b'') \otimes
\mathbf{B}(b'',b''')
 \otimes N(b''',c) \\
\displaystyle{\coprod_{b,b',b'' \in \mathbf{B}}} M(a,b) \otimes
\mathbf{B}(b,b') \otimes \mathbf{B}(b',b'')
 \otimes N(b'',c) \\
\displaystyle{\coprod_{b,b' \in \mathbf{B}}} M(a,b) \otimes \mathbf{B}(b,b')
\otimes N(b',c) \\
\displaystyle{\coprod_{b \in \mathbf{B}}} M(a,b) \otimes
N(b,c),
\\};
\path[->]
(m-1-1) edge (m-2-1)
([xshift=12] m-2-1.south) edge ([xshift=12] m-3-1.north)
([xshift=8] m-3-1.north) edge ([xshift=8] m-2-1.south)
([xshift=4] m-2-1.south) edge ([xshift=4] m-3-1.north)
(m-3-1) edge (m-2-1)
([xshift=-4] m-2-1.south) edge ([xshift=-4] m-3-1.north)
([xshift=-12] m-2-1.south) edge ([xshift=-12] m-3-1.north)
([xshift=-8] m-3-1.north) edge ([xshift=-8] m-2-1.south)

([xshift=8] m-3-1.south) edge ([xshift=8] m-4-1.north)
([xshift=4] m-4-1.north) edge ([xshift=4] m-3-1.south)
(m-3-1) edge (m-4-1)
([xshift=-8] m-3-1.south) edge ([xshift=-8] m-4-1.north)
([xshift=-4] m-4-1.north) edge ([xshift=-4] m-3-1.south)
([xshift=4] m-4-1.south) edge ([xshift=4] m-5-1.north)
(m-5-1) edge (m-4-1)
([xshift=-4] m-4-1.south) edge ([xshift=-4] m-5-1.north)
;
\end{tikzpicture}%
\]
where the face maps are given by the action of $\mathbf{B}$ on the
bimodules, and the degeneracy maps by the unit maps for $\mathbf{B}$.
We should therefore expect the composition of bimodules for enriched
\icats{} to be given by the colimit of a simplicial object analogous
to this. 

On the other hand, in \cite{nmorita} we defined the tensor
product of bimodules for associative algebras as an operadic left Kan
extension. This procedure has a natural generalization to the
many-object setting, which gives a precise definition of the composite
of two bimodules. We'll now introduce this, and then show that this
operadic Kan extension is in fact given by taking the expected
analogue of the colimit above.

\begin{definition}
  Given spaces $X_{0},X_{1},X_{2}$, we let $\Dop_{X_{0},X_{1},X_{2}}
  \to \Dop_{/[2]}$ be the left fibration associated to the functor
  $\Dop_{/[2]} \to \mathcal{S}$ obtained by right Kan extension from
  the functor $\{(0),(1),(2)\} \to \mathcal{S}$ sending $(i)$ to
  $X_{i}$.
\end{definition}
The composite $\Dop_{X_{0},X_{1},X_{2}}
  \to \Dop_{/[2]} \to \Dop$ is a double \icat{} by Lemma~\ref{lem:DopXndouble}.

\begin{remark}
  A $\Dop_{X_{0},X_{1},X_{2}}$-algebra in a monoidal \icat{}
  $\mathcal{V}$ can be interpreted as the data of:
  \begin{itemize}
  \item three $\mathcal{V}$-\icats{} $\mathcal{C}_{i}$ with $X_{i}$ as
    space of objects ($i = 0,1,2$),
  \item three bimodules: for $0 \leq i < j \leq 2$, a
    $\mathcal{C}_{i}$-$\mathcal{C}_{j}$-bimodule $M_{ij}$
  \item a $\mathcal{C}_{1}$-bilinear map $M_{01} \otimes M_{12} \to
    M_{02}$, i.e. given $x_{i} \in X_{i}$ we have maps \[M_{01}(x_{0},x_{1})
    \otimes M_{12}(x_{1},x_{2}) \to M_{02}(x_{0},x_{2}),\] compatible
    with the action of $\mathcal{C}_{1}$.
  \end{itemize}
  We want to restrict ourselves to the case where this map exhibits
  $M_{02}$ as the tensor product or composite $M_{01}
  \otimes_{\mathcal{C}_{1}} M_{12}$. As in \cite{nmorita}, we do this
  by considering only those $\Dop_{X_{0},X_{1},X_{2}}$ that arise as
  the left operadic Kan extensions of algebras for a subcategory of
  $\Dop_{X_{0},X_{1},X_{2}}$:
\end{remark}

\begin{definition}
  Recall that a map $\phi \colon [n] \to [m]$ in $\simp$ is said to be
  \emph{cellular} if $\phi(i+1) - \phi(i) \leq 1$ for all $i$, and
  that we write $\Lopn$ for the full subcategory of $\Dopn$ spanned by
  the cellular maps; this is a \gnsiopd{} by \cite{nmorita}*{Lemma 5.5}. We define $\Lop_{X_{0},X_{1},X_{2}}$ by the
  pullback square
  \nolabelcsquare{\Lop_{X_{0},X_{1},X_{2}}}{\Dop_{X_{0},X_{1},X_{2}}}{\Lop_{/[2]}}{\Dop_{/[2]}.}
\end{definition}
This is a pullback square in \gnsiopds{}, so
$\Lop_{X_{0},X_{1},X_{2}}$ is a \gnsiopd{}. Moreover, the inclusion
\[\tau_{X_{0},X_{1},X_{2}} \colon \Lop_{X_{0},X_{1},X_{2}} \to
\Dop_{X_{0},X_{1},X_{2}}\] is extendable in the sense of
Definition~\ref{defn:extendable} by
Proposition~\ref{propn:LopXnDopXnextble} (i.e. operadic left Kan
extensions along this map can be described using \gnsiopds{}), and
the \gnsiopd{} $\Lop_{X_{0},X_{1},X_{2}}$ is equivalent to the pushout
$\Dop_{X_{0},X_{1}} \amalg_{\Dop_{X_{1}}} \Dop_{X_{1},X_{2}}$ of
\gnsiopds{} by Corollary~\ref{cor:Lnopeq}.

This implies that if $\mathcal{V}$ is a monoidal \icat{} compatible with small
colimits, the restriction
\[
\begin{split}
\tau^{*}_{X_{0},X_{1},X_{2}} \colon \Alg_{\Dop_{X_{0},X_{1},X_{2}}}(\mathcal{V})
& \to \Alg_{\Lop_{X_{0},X_{1},X_{2}}}(\mathcal{V}) \\ &\simeq
\Alg_{\Dop_{X_{0},X_{1}}}(\mathcal{V})
\times_{\Alg_{\Dop_{X_{1}}}(\mathcal{V})}
\Alg_{\Dop_{X_{1},X_{2}}}(\mathcal{V})  
\end{split}
\] has a fully faithful left
adjoint $\tau_{X_{0},X_{1},X_{2},!}$ for all spaces
$X_{0},X_{1},X_{2}$. If $\mathcal{C}_{i}$ is a $\mathcal{V}$-\icat{}
with space of objects $X_{i}$ for $i = 0,1,2$, and we have a
$\mathcal{C}_{0}$-$\mathcal{C}_{1}$-bimodule $M$ and a
$\mathcal{C}_{1}$-$\mathcal{C}_{2}$-bimodule $N$, the \emph{composite}
$\mathcal{C}_{0}$-$\mathcal{C}_{2}$-bimodule $M
\otimes_{\mathcal{C}_{1}} N$ is the restriction to
$\Dop_{X_{0},X_{2}}$ of the $\Dop_{X_{0},X_{1},X_{2}}$-algebra
obtained by applying $\tau_{X_{0},X_{1},X_{2},!}$ to the
$\Lop_{X_{0},X_{1},X_{2}}$-algebra corresponding to $M$ and $N$.

\begin{remark}
  Let $\BMcat^{2}(\mathcal{V}) \to \mathcal{S}^{\times 3}$ be the
  Cartesian fibration associated to the functor $(\mathcal{S}^{\times
    3})^{\op} \to \CatI$ that sends $(X,Y,Z)$ to $\Alg_{\Dop_{X,Y,Z}}(\mathcal{V})$.
  The left adjoints $\tau_{X_{0},X_{1},X_{2},!}$ combine to give a
  fully faithful left adjoint to the restriction
  \[ \BMcat^{2}(\mathcal{V}) \to \BMcat(\mathcal{V})
  \times_{\Algcat(\mathcal{V})} \BMcat(\mathcal{V}).\]
  Combining this with the appropriate projection $\BMcat^{2}(\mathcal{V}) \to
  \BMcat(\mathcal{V})$ we get a composition functor
  \[\BMcat(\mathcal{V}) \times_{\Algcat(\mathcal{V})}
  \BMcat(\mathcal{V}) \to \BMcat(\mathcal{V}).\]
\end{remark}

Now we want to see that this composition of bimodules is given by
forming the expected colimit. The key observation is the following:
\begin{proposition}
  Given spaces $X,Y,Z$ and points $x \in X$ and $z \in Z$,
  let $\mathcal{X}_{x,z}$ be the \icat{} determined by the pullback
  square
  \nolabelcsquare{\mathcal{X}_{x,z}}{(\Lop_{X,Y,Z})^{\txt{act}}_{/(x,z)}}{\Dop}{(\Lop_{/[2]})^{\txt{act}}_{/(0,2)},}
  where the bottom horizontal map sends $[n]$ to the object
  $(0,1,\ldots,1,2)$ with $n+1$ copies of $1$. Then:
  \begin{enumerate}[(i)]
  \item There is an equivalence $\mathcal{X}_{x,z} \isoto \Dop_{Y}$.
  \item The top horizontal map $\mathcal{X}_{x,z} \to
    (\Lop_{X,Y,Z})^{\txt{act}}_{/(x,z)}$ is cofinal.
  \end{enumerate}
\end{proposition}
\begin{proof}
  The projection $(\Lop_{X,Y,Z})^{\txt{act}}_{/(x,z)} \to
  (\Lop_{/[2]})^{\txt{act}}_{/(0,2)}$ is a left fibration by
  Lemma~\ref{lem:LopXnactXifib}(i), hence so is the pullback
  $\mathcal{X}_{x,z} \to \Dop$. The functor taking a space $S$ to the
  left fibration $\Dop_{S} \to \Dop$ is right adjoint to the functor
  taking a left fibration over $\Dop$ to its fibre at $[0]$. Since the
  fibre of $\mathcal{X}_{x,y}$ at $[n]$ is $X_{/x} \times Y^{\times
    (n+1)} \times Z_{/z} \simeq Y^{\times (n+1)}$, the unit of this
  adjunction gives a map of left fibrations $\mathcal{X}_{x,z} \to
  \Dop_{Y}$. To see that this is an equivalence it suffices to show
  that the induced map on fibres at any $[n] \in \Dop$ is an
  equivalence, which is clear. This proves (i).

  The bottom horizontal map $\Dop \to
  (\Lop_{/[2]})^{\txt{act}}_{/(0,2)}$ is cofinal by
  \cite{nmorita}*{Lemma 5.7}. Combining \cite{HTT}*{Proposition
    4.1.2.15} with \cite{HTT}*{Remark 4.1.2.10} we know that the
  pullback of a cofinal map along a coCartesian fibration is cofinal,
  which proves (ii).
\end{proof}

From the definition of left operadic Kan extensions it therefore
follows that if $\mathcal{A}$, $\mathcal{B}$, and $\mathcal{C}$ are
$\mathcal{V}$-\icats{} with spaces of objects $X$, $Y$, and $Z$,
respectively, and we have an $\mathcal{A}$-$\mathcal{B}$-bimodule $M$
and a $\mathcal{B}$-$\mathcal{C}$-bimodule $N$, then the composite
$\mathcal{A}$-$\mathcal{C}$-bimodule $M \otimes_{\mathcal{B}} N$ is
given at $(x,z)$ by a $\Dop_{Y}$-indexed colimit we can informally
write as
\[ M \otimes_{\mathcal{B}} N \simeq \colim_{(y_{0},\ldots,y_{n}) \in
  \Dop_{Y}} M(x,y_{0}) \otimes \mathcal{B}(y_{0},y_{1}) \otimes \cdots
\otimes \mathcal{B}(y_{n-1},y_{n}) \otimes N(y_{n}, z).\] To relate
this to the expected geometric realization, we need a technical
observation:
\begin{proposition}\label{propn:doublecolimcoCart}
  Let $\mathcal{I}$ be an \icat{} and $p \colon \mathcal{I} \to \CatI$
  a functor with $\mathcal{K} \to \mathcal{I}$ its associated
  coCartesian fibration. Suppose $q \colon \mathcal{K} \to \mathcal{D}$
  is a functor such that for each $\alpha \in \mathcal{I}$ the diagram
  $q_{\alpha} \colon p(\alpha) \simeq \mathcal{K}_{\alpha}
  \to \mathcal{D}$ has a colimit; by \cite{HTT}*{Proposition 4.2.2.7}
  there exists an (essentially unique) map $q_{+} \colon
  \mathcal{K}_{+} \to \mathcal{D}$, where \[\mathcal{K}_{+} :=
  \mathcal{K}\times \Delta^{1} \amalg_{\mathcal{K} \times \{1\}}
  \mathcal{\mathcal{I}},\] that restricts to $q$ on $\mathcal{K}$ and
  to a colimit of $q_{\alpha}$ on $p(\alpha)^{\triangleright}
  \simeq \mathcal{K}_{+} \times_{\mathcal{I}}\{\alpha\}$. Then the
  maps \[\mathcal{D}_{q/} \leftarrow \mathcal{D}_{q_{+}/} \to
  \mathcal{D}_{q_{+}|_{\mathcal{I}}/}\] are trivial fibrations.
\end{proposition}
\begin{proof}
  The map $\mathcal{K} \times \{1\} \hookrightarrow \mathcal{K} \times
  \Delta^{1}$ is right anodyne by \cite{HTT}*{Corollary 2.1.2.7}, so
  the pushout $\mathcal{I} \to \mathcal{K}_{+}$ is also right anodyne
  and thus cofinal by \cite{HTT}*{Proposition 4.1.1.3}. Therefore
  $\mathcal{D}_{q_{+}/} \to \mathcal{D}_{q_{+}|_{\mathcal{I}}/}$ is a
  trivial fibration by \cite{HTT}*{Proposition 4.1.1.7}. On the other
  hand, since $\pi \colon \mathcal{K} \to \mathcal{I}$ is a coCartesian
  fibration, for each $i \in \mathcal{I}$ the inclusion
  \[\mathcal{K}_{i} \hookrightarrow \mathcal{K}_{/i} := \mathcal{K}
  \times_{\mathcal{I}} \mathcal{I}_{/i} \] is cofinal --- this follows
  by applying \cite{HTT}*{Theorem 4.1.3.1}, since for every object $(k,
  f\colon \pi(k) \to i) \in \mathcal{K}_{/i}$ the fibre
  $(\mathcal{K}_{i})_{(k,f)/}$ has an initial object given by the
  coCartesian map $k \to f_{!}k$ and so is weakly contractible. Thus
  $q_{+}$ is a left Kan extension of $q$ along $\mathcal{K}
  \hookrightarrow \mathcal{K}_{+}$, and hence $\mathcal{D}_{q_{+}/} \to
  \mathcal{D}_{q/}$ is a trivial fibration by \cite{HTT}*{Lemma
    4.3.2.7}.
\end{proof}

\begin{corollary}
  Let $q \colon \mathcal{K} \to \mathcal{D}$ be as above. If 
  the diagram $q$ has a colimit, we have an equivalence
  \[ \colim_{\mathcal{K}} q \simeq \colim_{\alpha \in \mathcal{I}}
  \colim_{p(\alpha)} q_{\alpha},\] where the functoriality in $\alpha$
  of the colimits over $p(\alpha)$ comes from the diagram $q_{+}$.
\end{corollary}
\begin{proof}
  Since the maps in Proposition~\ref{propn:doublecolimcoCart} are
  trivial fibrations compatible with the projections to $\mathcal{D}$,
  the colimit of $q$, which is the initial object of
  $\mathcal{D}_{q/}$, must project to an object of $\mathcal{D}$ that
  is equivalent to the projection of the initial object of
  $\mathcal{D}_{q_{+}|_{\mathcal{I}}/}$, which is a colimit of the
  diagram $\alpha \mapsto \colim_{p(\alpha)} q_{\alpha}$ induced by
  $q_{+}$.
\end{proof}

Since $\Dop_{Y} \to \Dop$ is a left fibration, we can apply this to
our $\Dop_{Y}$-indexed colimit for $(M \otimes_{\mathcal{B}} N)(x,z)$ to
conclude that, as we expected, this is equivalent to the geometric
realization of a simplicial diagram with $n$th term
  \[ \colim_{(y_{0},\ldots,y_{n}) \in Y^{\times (n+1)}} M(a,y_{0})
  \otimes \mathcal{B}(y_{0},y_{1}) \otimes \cdots \otimes
  \mathcal{B}(y_{n-1},y_{n}) \otimes N(y_{n}, c).\]

\section{The Double $\infty$-Category of Enriched $\infty$-Categories}\label{sec:double}
Now we get to the meat of this paper --- in this section we'll
construct a double \icat{} of $\mathcal{V}$-\icats{}, in the form of a
simplicial \icat{} whose value at $[0]$ is $\Algcat(\mathcal{V})$ and
at $[1]$ is $\BMcat(\mathcal{V})$, with the composition 
\[ \BMcat(\mathcal{V}) \times_{\Algcat(\mathcal{V})}
\BMcat(\mathcal{V}) \to \BMcat(\mathcal{V}) \] given by the
construction we discussed in the previous section.

The basic objects we consider are again the natural many-object
versions of those we used in \cite{nmorita}:
\begin{definition}
  Given spaces $X_{0},\ldots,X_{n}$ we define
  $\Dop_{X_{0},\ldots,X_{n}} \to \Dopn$ to be the coCartesian
  fibration associated to the functor $\Dopn \to \mathcal{S}$ obtained
  by right Kan extension from the functor $\{(0),\ldots,(n)\} \to
  \mathcal{S}$ that sends $(i)$ to $X_{i}$.
\end{definition}
\begin{lemma}\label{lem:DopXndouble}
  The composite $\Dop_{X_{0},\ldots,X_{n}} \to \Dopn \to \Dop$ is a
  double \icat{} for all spaces $X_{0},\ldots,X_{n}$. 
\end{lemma}
\begin{proof}
  The right Kan-extended functor $F\colon \Dopn \to \mathcal{S}$ clearly sends
  $(i_{0},\ldots,i_{n})$ to $X_{i_{0}} \times \cdots \times
  X_{i_{n}}$, and so is a $\Dopn$-category object in the sense of
  \cite{enr}*{Definition 3.5.2}, i.e. it satisfies the Segal condition
  \[ F(i_{0},\ldots,i_{n}) \isoto F(i_{0},i_{1}) \times_{F(i_{1})}
  \cdots \times_{F(i_{n-1}} F(i_{n-1},i_{n}).\]
  The composite coCartesian fibration $\DopXn \to \Dop$ is therefore a double
  \icat{} by \cite{enr}*{Proposition 3.5.4}.
\end{proof}

\begin{definition}
  If $\mathcal{C}$ is an \icat{}, let $\simp_{\mathcal{C}} \to \simp$
  be a Cartesian fibration associated to the functor $\simp^{\op} \to
  \CatI$ that is the right Kan extension of the functor $\{[0]\} \to
  \CatI$ sending $[0]$ to $\mathcal{C}$. (This functor sends $[n]$ to
  $\mathcal{C}^{\times (n+1)}$.)
\end{definition}
We now observe that the double \icats{}
$\DopXn$ combine to a functor
$\simp_{\mathcal{S}} \to \OpdInsg$:
\begin{definition}
  Let $i_{n} \colon \{0,\ldots,n\} \hookrightarrow \simp^{\op}_{/[n]}$
  be the functor that takes $j \in \{0,\ldots,n\}$ to the map $[0] \to [n]$ sending $0$
  to $j \in [n]$. (This is equivalent to the inclusion of the fibre of
  $\Dop_{/[n]}$ over $[0] \in \simp^{\op}$.) Right Kan extension along
  $i_{n}$ then determines a functor \[ \mathcal{S}^{\times (n+1)}
  \simeq \Fun(\{0,\ldots,n\}, \mathcal{S}) \to
  \Fun(\simp^{\op}_{/[n]}, \mathcal{S}),\] which is moreover natural
  in $[n] \in \simp^{\op}$. Using the equivalence between functors to
  $\mathcal{S}$ and left fibrations, we get a functor from
  $\mathcal{S}^{\times (n+1)}$ to left fibrations over
  $\simp^{\op}_{/[n]}$ that takes $(X_{0},\ldots,X_{n})$ to
  $\Dop_{X_{0},\ldots,X_{n}} \to \Dopn$. Since the resulting composite
  maps $\simp^{\op}_{X_{0},\ldots,X_{n}} \to \simp^{\op}$ are double
  \icats{}, we get functors $\mathcal{S}^{\times (n+1)} \to \OpdInsg$
  taking $(X_{0},\ldots,X_{n})$ to $\Dop_{X_{0},\ldots,X_{n}} \to
  \Dop$, natural in $n$. Using the universal property of
  $\simp_{\mathcal{S}}$ from \cite{freepres}*{Proposition 7.3} this
  corresponds to a functor $\simp_{\mathcal{S}} \to \OpdInsg$.
\end{definition}

\begin{definition}
  For any \gnsiopd{} $\mathcal{O}$, we let \[\ofALGcat'(\mathcal{O}) \to \simp_{\mathcal{S}}\] be a
  Cartesian fibration associated to the functor
  $(\simp_{\mathcal{S}})^{\op} \to \CatI$ given by
  \[(X_{0},\ldots,X_{n}) \mapsto
  \Alg_{\DopXn}(\mathcal{O}).\] Then we
  define $\ofALGcat(\mathcal{O}) \to \simp^{\op}$ to be a
  coCartesian fibration corresponding to the composite Cartesian
  fibration $\ofALGcat'(\mathcal{O}) \to \simp_{\mathcal{S}} \to
  \simp.$ In particular, over $[n] \in \simp^{\op}$ we have a
  Cartesian fibration $\ofALGcat(\mathcal{O})_{[n]}\to
  \mathcal{S}^{\times (n+1)}$.
\end{definition}

\begin{definition}
  For spaces $X_{0},\ldots,X_{n}$, we define a \gnsiopd{}
  $\LopXn$ by the pullback diagram
  \csquare{\LopXn}{\DopXn}{\Lbrnop}{\Dopn.}{\tau_{(X_{0},\ldots,X_{n})}}{}{}{\tau_{n}}
\end{definition}

We say that a $\DopXn$-algebra is \emph{composite} if it is the
operadic left Kan extension of its restriction to $\LopXn$. To
understand these operadic left Kan extensions we must check that the
map $\LopXn \to \DopXn$ is \emph{extendable}, in the sense of
Definition~\ref{defn:extendable}.

To prove this we first make the following technical observation:
\begin{lemma}\label{lem:LopXnactXifib}
  Suppose given spaces $X_{0},\ldots,X_{n}$, a morphism $\xi \colon [m] \to
  [n]$ in $\simp$, and $\Xi \in \Dop_{X_{0},\ldots,X_{n}}$
  over $\xi \in \Dopn$. Then:
  \begin{enumerate}[(i)]
  \item The projection
  \[ (\Lop_{X_{0},\ldots,X_{n}})^{\txt{act}}_{/\Xi} \to
  (\Lbrnop)^{\txt{act}}_{/\xi} \] is a left fibration.
\item For any cellular morphism $\eta \colon [k] \to [n]$ and active
  morphism $\phi \colon [m] \to [k]$ in $\simp$ such that $\xi =
  \eta\phi$, the fibre of this projection at $(\phi,\eta)$ is the
  pullback \nolabelcsquare{
    ((\Lop_{X_{0},\ldots,X_{n}})^{\txt{act}}_{/\Xi})_{(\phi,\eta)}}{\prod_{j
      = 0}^{k} X_{\eta(j)}}{\{\Xi\}}{\prod_{i = 0}^{m} X_{\xi(j)},}
  where the right vertical map sends $(p_{0},\ldots,p_{m})$ to
  $(p_{\phi(0)},\ldots,p_{\phi(m)})$.
  \end{enumerate}
\end{lemma}
\begin{proof}
  First consider the commutative diagram
  \[
  \begin{tikzcd}
    (\LopXn)^{\txt{act}}_{/\Xi} \arrow{r} \arrow{d} & (\DopXn)^{\txt{act}}_{/\Xi}
    \arrow{d} \\
    (\LopXn)^{\txt{act}} \arrow{r} \arrow{d}  &
    (\DopXn)^{\txt{act}} \arrow{d} \\
    (\Lbrnop)^{\txt{act}} \arrow{r} & (\Dopn)^{\txt{act}}.
  \end{tikzcd}
  \]
  Here the top square is Cartesian by the definition of
  $(\LopXn)^{\txt{act}}_{/\Xi}$, and it follows immediately from the
  definition of $\LopXn$ that the bottom square is also
  Cartesian. Thus the composite square is also Cartesian.

  Now consider the diagram
  \[
  \begin{tikzcd}
    (\LopXn)^{\txt{act}}_{/\Xi} \arrow{r} \arrow{d} & (\DopXn)^{\txt{act}}_{/\Xi}
    \arrow{d} \\
    (\Lbrnop)^{\txt{act}}_{/\xi} \arrow{r} \arrow{d}  &
    (\Dopn)^{\txt{act}}_{/\xi} \arrow{d} \\
    (\Lbrnop)^{\txt{act}} \arrow{r} & (\Dopn)^{\txt{act}}.
  \end{tikzcd}
  \]
  Here the bottom square is Cartesian by the definition of
  $(\Lbrnop)^{\txt{act}}_{/\xi}$, hence since the composite square is
  Cartesian so is the top square.

  The projection $\DopXn \to \Dopn$ is by definition a left fibration,
  hence so is the restriction $(\DopXn)^{\txt{act}} \to
  (\Dopn)^{\txt{act}}$ to the active maps, since this can be described
  as the pullback along $\Dop_{\txt{act}} \to \Dop$. The projection
  $(\DopXn)_{/\Xi}^{\txt{act}} \to (\Dnop)_{/\xi}^{\txt{act}}$ is
  therefore a left fibration by \cite{HTT}*{Proposition 2.1.2.1},
  hence so is the pullback $(\LopXn)^{\txt{act}}_{/\Xi} \to (\Lopn)^{\txt{act}}_{/\xi}$. This
  proves (i).

  (ii) then follows immediately from the definition of the left
  fibration $\DopXn \to \Dopn$.
\end{proof}

\begin{proposition}\label{propn:LopXnDopXnextble}
  For any spaces $X_{0},\ldots,X_{n}$, the inclusion $\LopXn \to
  \DopXn$ is extendable.
\end{proposition}
\begin{proof}
  We must show that for any $\Xi \in \DopXn$ (lying over $\xi \colon
  [m] \to [n]$), the map
  \[ (\LopXn)^{\txt{act}}_{/\Xi} \to \prod_{p = 1}^{m}
  (\LopXn)^{\txt{act}}_{/\rho_{p}^{*}\Xi} \]
  is cofinal. Consider the commutative square
   \nolabelcsquare{(\LopXn)^{\txt{act}}_{/\Xi}}{\prod_{p = 1}^{m}
  (\LopXn)^{\txt{act}}_{/\rho_{p}^{*}\Xi}}{(\Lbrnop)^{\txt{act}}_{/\xi}}{\prod_{p = 1}^{m}
  (\Lbrnop)^{\txt{act}}_{/\rho_{p}^{*}\xi.}}
  Here the vertical maps are left fibrations by
  Lemma~\ref{lem:LopXnactXifib}(i), and the bottom horizontal map is
  cofinal by \cite{nmorita}*{Proposition 5.6}. \cite{HTT}*{Proposition
    4.1.2.15} together with \cite{HTT}*{Remark 4.1.2.10} implies that
  the pullback of a cofinal map along a coCartesian fibration is
  cofinal, so to show that the top horizontal map is cofinal it's
  enough to prove that this square is Cartesian. Since the vertical
  maps are left fibrations, for this it suffices to show that the
  induced map on the fibres at any $(\phi,\eta) \in
  (\Lbrnop)^{\txt{act}}_{/\xi}$ is an equivalence, which is clear from
  the description of the fibres in
  Lemma~\ref{lem:LopXnactXifib}(ii).
\end{proof}

Applying \cite{nmorita}*{Theorem 4.40} (which itself is a minor
variation of \cite{HA}*{Theorem 3.1.2.3}), we get:
\begin{corollary}
  Suppose $\mathcal{V}$ is a monoidal \icat{} compatible with small
  colimits. Then for any spaces $X_{0},\ldots,X_{n}$ the
  restriction \[\tau_{X_{0},\ldots,X_{n}}^{*} \colon
  \Alg_{\DopXn}(\mathcal{V}) \to \Alg_{\LopXn}(\mathcal{V})\] has a
  fully faithful left adjoint $\tau_{X_{0},\ldots,X_{n},!}$.
\end{corollary}

\begin{definition}
  Suppose $\mathcal{V}$ is a monoidal \icat{} compatible with small
  colimits. We say a $\DopXn$-algebra $\mathcal{M}$ in $\mathcal{V}$ is
  \emph{composite} if it lies in the image of
  $\tau_{X_{0},\ldots,X_{n},!}$, or equivalently if the counit
  map $\tau_{X_{0},\ldots,X_{n},!}\tau_{X_{0},\ldots,X_{n}}^{*}\mathcal{M}
  \to \mathcal{M}$ is an equivalence.  
\end{definition}

\begin{definition}
  Suppose $\mathcal{V}$ is a monoidal \icat{} compatible with small
  colimits. Let $\fALGcat(\mathcal{V})$ denote the full subcategory of
  $\ofALGcat(\mathcal{V})$ spanned by the composite $\DopXn$-algebras
  for all spaces $X_{0},\ldots,X_{n}$.
\end{definition}

\begin{remark}
  The natural transformations $\LopXn \to \DopXn$ induce a map of
  Cartesian fibrations \[\tau_{n}^{*} \colon \ofALGcat(\mathcal{V})_{n}
  \to \ofALGcat^{\bbLambda}(\mathcal{V})_{n}\] over
  $\mathcal{S}^{\times (n+1)}$, where
  $\ofALGcat^{\bbLambda}(\mathcal{V})_{n} \to \mathcal{S}^{\times
    (n+1)}$ is a Cartesian fibration associated to the functor
  $(\mathcal{S}^{\times (n+1)})^{\op} \to \CatI$ that sends
  $(X_{0},\ldots,X_{n})$ to $\Alg_{\LopXn}(\mathcal{V})$.  If
  $\mathcal{V}$ is compatible with small colimits, the fibrewise left
  adjoints $\tau_{X_{0},\ldots,X_{n},!}$ then combine to give a left
  adjoint $\tau_{n,!} \colon \ofALGcat^{\bbLambda}(\mathcal{V})_{n}
  \to \ofALGcat(\mathcal{V})_{n}$ by \cite[Proposition 7.3.2.6]{HA},
  and we can define $\fALGcat(\mathcal{V})_{n}$ to be the image of
  $\tau_{n,!}$. In particular, the projection
  $\fALGcat(\mathcal{V})_{[n]} \to \mathcal{S}^{\times (n+1)}$ is
  still a Cartesian fibration.
\end{remark}

Next we need to show that the projection $\fALGcat(\mathcal{V}) \to
\Dop$ is a coCartesian fibration. This will follow from an
extension of \cite{nmorita}*{Proposition 5.16}:
\begin{definition}
  Recall that for $\phi \colon [m] \to [n]$ in $\simp$, we say that a
  morphism $\alpha \colon [k] \to [n]$ is \emph{$\phi$-cellular} if 
  \begin{enumerate}[(1)]
  \item for $i$ such that $\alpha(i) < \phi(0)$ we have $\alpha(i+1)
    \leq \alpha(i)+1$,
  \item for $i$ such that $\phi(j) \leq \alpha(i) < \phi(j+1)$ we have
    $\alpha(i+1) \leq \phi(j+1)$,
  \item for $i$ such that $\alpha(i) \geq \phi(m)$ we have
    $\alpha(i+1) \leq \alpha(i)+1$.
  \end{enumerate}
  We write $\Lopn[\phi]$ for the full subcategory of $\Dop_{n}$
  spanned by the $\phi$-cellular maps to $[n]$, and for spaces
  $X_{0},\ldots,X_{n}$ we define $\LopXn[\phi]$ by the pullback square
\nolabelcsquare{\LopXn[\phi]}{\DopXn}{\Lopn[\phi]}{\Dopn.}
\end{definition}
\begin{proposition}\label{propn:phicellcof}
  For any $\phi \colon [m] \to [n]$ and $\Gamma \in
  \Dop_{X_{\phi(0)},\ldots,X_{\phi(m)}}$ over $\phi \colon [k] \to [m]
  \in \Dop_{/[m]}$, the map
  \[ (\Lop_{X_{\phi(0)},\ldots,X_{\phi(m)}})^{\txt{act}}_{/\Gamma} \to
  (\Lop_{X_{0},\ldots,X_{n}}[\phi])^{\txt{act}}_{/\phi_{*}\Gamma} \]
  is cofinal.
\end{proposition}
\begin{proof}
  Consider the commutative square
  \nolabelcsquare{(\Lop_{X_{\phi(0)},\ldots,X_{\phi(m)}})^{\txt{act}}_{/\Gamma}}{(\Lop_{X_{0},\ldots,X_{n}}[\phi])^{\txt{act}}_{/\phi_{*}\Gamma}}{(\Lop_{/[m]})^{\txt{act}}_{/\gamma}}{(\Lbrnop[\phi])^{\txt{act}}_{/\phi\gamma}.}
  The proof of Lemma~\ref{lem:LopXnactXifib} clearly extends to the
  $\phi$-cellular case, so the vertical maps here are left fibrations
  and the bottom horizontal map is cofinal by
  \cite{nmorita}*{Proposition 5.16}. \cite{HTT}*{Proposition 4.1.2.15}
  together with \cite{HTT}*{Remark 4.1.2.10} implies that the pullback
  of a cofinal map along a coCartesian fibration is cofinal, so to
  show that the top horizontal map is cofinal it's enough to prove
  that this square is Cartesian. Since the vertical maps are left
  fibrations, for this it suffices to show that the induced map on the
  fibres at any object of $(\Lop_{/[m]})^{\txt{act}}_{/\gamma}$ is an
  equivalence, which is clear from the description of the fibres in
  Lemma~\ref{lem:LopXnactXifib}(ii).
\end{proof}

\begin{corollary}\label{cor:ALGcatcoCart}
  The restricted projection $\fALGcat(\mathcal{V}) \to \Dop$ is a
  coCartesian fibration.
\end{corollary}
\begin{proof}
  This follows, using Proposition~\ref{propn:phicellcof}, by exactly
  the same proof as that of \cite{nmorita}*{Corollary 5.17}.
\end{proof}

It follows that $\fALGcat(\mathcal{V}) \to \Dop$ determines a functor
$\Dop \to \CatI$. We want to show that $\fALGcat(\mathcal{V})$ is a
double \icat{}, i.e. that this functor satisfies the Segal
condition. We'll deduce this from the following observation:

\begin{proposition}\label{propn:propereqs}
  Let $\mathfrak{P} = (\mathcal{C}, S, \{K_{\alpha}^{\triangleleft}\to
  \mathcal{C}\})$ be a categorical pattern, in the sense of
  \cite{HA}*{B.0.19}, and suppose $f \colon A \to B$ is a trivial
  cofibration in the model category $(\sSet^{+})_{\mathfrak{P}}$ of
  \cite{HA}*{Theorem B.0.20} of one of the following kinds:
  \begin{enumerate}[(a)]
  \item $(\Lambda^{n}_{i}, T) \hookrightarrow (\Delta^{n}, T)$, where
    $T$ consists of the degenerate edges together with $(i-1) \to i$,
    for all $0 < i < n$.
  \item $(\partial \Delta^{n} \star K_{\alpha}, T) \hookrightarrow
    (\Delta^{n} \star K_{\alpha}, T)$, where $T$ consists of the
    non-degenerate edges together with all edges in $K_{\alpha}$ and
    all maps $n \to k$ for $k \in K_{\alpha}$.
  \end{enumerate}
  Then for any pullback diagram \csquare{X \times_{B}
    A}{X}{A}{B}{f'}{g'}{g}{f} in $(\sSet^{+})_{\mathfrak{P}}$ where
  $g$ is a fibration, the morphism $f' \colon X \times_{B} A \to X$ is
  a weak equivalence.
\end{proposition}
\begin{proof}
  In the case (a) this follows from \cite[Lemma 2.4.4.6]{HA} and in
  the case (b) it holds by the same argument as in the proofs of
  \cite[Lemmas 2.4.4.4 and 2.4.4.5]{HA}.
\end{proof}

\begin{definition}
  As in \cite{nmorita}*{\S 5.2}, we write $\simp^{\amalg,\op}_{/[n]}$
  for the ordinary colimit $\Dop_{/[1]} \amalg_{\Dop_{/[0]}} \cdots
  \amalg_{\Dop_{/[0]}} \Dop_{/[1]}$ in (marked) simplicial sets (over
  $\Dop$), with $n$ copies of $\Dop_{/[1]}$. For any Kan complexes
  $X_{0}, \ldots, X_{n}$, we then define
  $\simp^{\amalg,\op}_{X_{0},\ldots,X_{n}}$ by the (ordinary) pullback
  \nolabelcsquare{\simp^{\amalg,\op}_{X_{0},\ldots,X_{n}}}{\LopXn}{\simp^{\amalg,\op}_{/[n]}}{\Lbrnop}
  in $(\sSet^{+})_{\mathfrak{O}_{\txt{ns}}^{\txt{gen}}}$.
\end{definition}

\begin{lemma}\label{lem:amalgXcolim}
  For any Kan complexes $X_{0},\ldots,X_{n}$, the marked simplicial
  set $\simp^{\amalg,\op}_{X_{0},\ldots,X_{n}}$ is equivalent to the
  homotopy colimit $\simp^{\op}_{X_{0},X_{1}}
  \amalg_{\simp^{\op}_{X_{1}}} \cdots \amalg_{\simp^{\op}_{X_{n-1}}}
  \simp^{\op}_{X_{n-1},X_{n}}$, i.e. the corresponding iterated
  pushout in the \icat{} $\OpdInsg$.
\end{lemma}
\begin{proof}
  Pullbacks in $\sSet^{+}$ preserve colimits, so since
  $\simp^{\amalg,\op}$ is a colimit we may identify
  $\simp^{\amalg,\op}_{X_{0},\ldots,X_{n}}$ with the corresponding (ordinary)
  colimit. But since this colimit can be written as an iterated
  pushout along cofibrations, this colimit is a homotopy colimit.
\end{proof}

\begin{corollary}\label{cor:Lnopeq} 
  For any spaces $X_{0},\ldots,X_{n}$, the inclusion
  \[\simp^{\amalg,\op}_{X_{0},\ldots,X_{n}} \hookrightarrow
  \LopXn \]
  is a trivial cofibration in
  $(\sSet^{+})_{\mathfrak{O}_{\txt{ns}}^{\txt{gen}}}$.
\end{corollary}
\begin{proof}
  The proof of \cite[Proposition 5.10]{nmorita} implies that the
  inclusion $\simp^{\amalg,\op}_{/[n]} \hookrightarrow \Lbrnop$ is a
  transfinite composite of pushouts of the morphisms described in
  Proposition~\ref{propn:propereqs}. Since $\sSet^{+}$ is locally
  Cartesian closed, it follows that for any spaces $X_{0}$, \ldots,
  $X_{n}$, the inclusion $\simp^{\amalg,\op}_{X_{0},\ldots,X_{n}}
  \hookrightarrow \LopXn$ is a transfinite
  composite of pushouts along pullbacks of such maps. Since the
  projection $\LopXn \to \Lbrnop$ is a fibration
  in this model structure, it follows from
  Proposition~\ref{propn:propereqs} that this map is a trivial
  cofibration.
\end{proof}

\begin{corollary}\label{cor:Segcond}
  Let $\mathcal{V}$ be a monoidal \icat{} compatible with small
  colimits. Then the Segal map
  \[ \fALGcat(\mathcal{V})_{n} \to \fALGcat(\mathcal{V})_{1}
  \times_{\fALGcat(\mathcal{V})_{0}} \cdots
  \times_{\fALGcat(\mathcal{V})_{0}} \fALGcat(\mathcal{V})_{1}\]
  is an equivalence of \icats{}.
\end{corollary}
\begin{proof}
  This is a map of Cartesian fibrations over $\mathcal{S}^{\times
    (n+1)}$, so it suffices to show that for all spaces
  $X_{0},\ldots,X_{n}$ the induced map on fibres over
  $(X_{0},\ldots,X_{n})$ is an equivalence. But this map can be
  identified with the composite
  \[ 
  \begin{split}
    (\fALGcat(\mathcal{V})_{n})_{(X_{0},\ldots,X_{n})} &  \to
  \Alg_{\LopXn}(\mathcal{V}) \\ & \to
  \Alg_{\simp^{\amalg,\op}_{X_{0},\ldots,X_{n}}}(\mathcal{V}) \\
&  \to
  \Alg_{\Dop_{X_{0},X_{1}}}(\mathcal{V})
  \times_{\Alg_{\Dop_{X_{1}}}(\mathcal{V})}\cdots\times_{\Alg_{\Dop_{X_{n-1}}}(\mathcal{V})}   \Alg_{\Dop_{X_{n-1},X_{n}}}(\mathcal{V}),
  \end{split}
  \]
  where the first map is an equivalence by definition, the second by
  Corollary~\ref{cor:Lnopeq}, and the third by
  Lemma~\ref{lem:amalgXcolim}. 
\end{proof}

Combining Corollary~\ref{cor:Segcond} with
Corollary~\ref{cor:ALGcatcoCart}, we have proved:
\begin{theorem}
  Let $\mathcal{V}$ be a monoidal \icat{} compatible with small
  colimits. Then the projection $\fALGcat(\mathcal{V}) \to \Dop$ is a
  double \icat{}.
\end{theorem}

\begin{definition}
  We say a $\DopXn$-algebra $\mathcal{M}$ in $\mathcal{V}$ is
  \emph{complete} if for each $i = 0,\ldots,n$ the
  $\simp^{\op}_{X_{i}}$-algebra $\sigma^{*}_{i}\mathcal{M}$ is
  complete, where $\sigma_{i} \colon [0] \to [n]$ is the map sending
  $0$ to $i$. We define $\fCAT(\mathcal{V})$ to be the full
  subcategory of $\fALGcat(\mathcal{V})$ spanned by the complete
  composite $\DopXn$-algebras for all spaces $X_{0},\ldots,X_{n}$.
\end{definition}

\begin{corollary}
  The projection $\fCAT(\mathcal{V}) \to \Dop$ is a double \icat{}.
\end{corollary}
\begin{proof}
  To see that $\fCAT(\mathcal{V}) \to \Dop$ is coCartesian, it
  suffices to observe that given a complete $\DopXn$-algebra $M \colon \DopXn \to
  \mathcal{V}^{\otimes}$ and a morphism $\phi
  \colon [m] \to [n]$ in $\simp$, the target
  $\phi^{*}M \colon \Dop_{X_{\phi(0)},\ldots,X_{\phi(m)}} \to
  \mathcal{V}^{\otimes}$ of the coCartesian morphism over $\phi$ with
  source $M$ is also complete, since
  $\sigma_{i}^{*}\phi^{*}M \simeq \sigma_{\phi(i)}^{*}M$ and this is
  complete.
  
  To see that it is moreover a double \icat{}, observe that
 \[ \fCAT(\mathcal{V})_{1} \simeq \CatIV
  \times_{\fALGcat(\mathcal{V})_{0}} \fALGcat(\mathcal{V})_{1}
  \times_{\fALGcat(\mathcal{V})_{0}} \CatIV,\]
  and so under the
  identification of Corollary~\ref{cor:Segcond}, the full subcategory
  $\fCAT(\mathcal{V})_{n}$ of $\fALGcat(\mathcal{V})_{n}$ precisely corresponds to the iterated fibre  product
  \[ \fCAT(\mathcal{V})_{1} \times_{\fCAT(\mathcal{V})_{0}} \cdots
  \times_{\fCAT(\mathcal{V})_{0}} \fCAT(\mathcal{V})_{1}.\qedhere\]
\end{proof}

\section{Natural Transformations}\label{sec:nattr}
In this section we consider the obvious definition of \emph{natural
transformations} between functors of enriched \icats{}. We then
use this to construct \icats{} of functors and show that these are the
underlying \icats{} of the internal Hom when this exists.

\begin{definition}
  We may regard the categories $[n]$ as (levelwise discrete) Segal
  spaces, and thus as $\mathcal{S}$-enriched \icats{} via the
  equivalence of \cite{enr}*{Theorem 4.4.7}. If $\mathcal{V}$ is a
  monoidal \icat{} compatible with small colimits, then
  $\Algcat(\mathcal{V})$ is tensored over $\Algcat(\mathcal{S})$ by
  \cite{enr}*{Corollary 4.3.17}, so for any $\mathcal{V}$-\icat{}
  $\mathcal{C}$ we have $\mathcal{V}$-\icats{} $\mathcal{C} \otimes
  [n]$\footnote{In fact, we may define $\mathcal{C} \otimes [n]$ as a
    $\mathcal{V}$-\icat{} provided only that $\mathcal{V}$ has an
    initial object and this is compatible with the tensor product, but
    we will not need this generality.}.
  If $f,g \colon \mathcal{C} \to \mathcal{D}$ are functors of
  $\mathcal{V}$-\icats{}, a \emph{natural transformation} from $f$ to
  $g$ is a functor $\eta \colon \mathcal{C} \otimes [1] \to
  \mathcal{D}$ such that $\eta \circ (\id_{\mathcal{C}} \otimes d_{1})
  \simeq f$ and $\eta \circ (\id_{\mathcal{C}} \otimes d_{0}) \simeq g$.
\end{definition}

Given this definition of natural transformations, there is an obvious
simplicial space that should be the \icat{} of $\mathcal{V}$-functors
between two $\mathcal{V}$-\icats{}:
\begin{definition}
  Suppose $\mathcal{C}$ and $\mathcal{D}$ are
  $\mathcal{V}$-\icats{}. We let $\Fun_{\mathcal{V}}(\mathcal{C},
  \mathcal{D})$ denote the simplicial space $\Dop \to \mathcal{S}$ sending
  $[n]$ to $\Map_{\CatIV}(\mathcal{C} \otimes [n], \mathcal{D})$.
\end{definition}

Our first goal in this section is to check that this is indeed a Segal
space, and that it's complete if the target is a complete $\mathcal{V}$-\icat{}:
\begin{proposition}\label{propn:[n]coSegal}
  Let $\mathcal{V}$ be a monoidal \icat{} compatible with small
  colimits, and let $\mathcal{C}$ and $\mathcal{D}$ be
  $\mathcal{V}$-\icats{}. 
  \begin{enumerate}[(i)]
  \item The simplicial space $\Fun_{\mathcal{V}}(\mathcal{C}, \mathcal{D})$
  is a Segal space.
  \item For any Segal space $X$ we have a natural equivalence
    \[ \Map_{\SegI}(X, \FunV(\mathcal{C}, \mathcal{D})) \simeq
    \Map_{\AlgCat(\mathcal{V})}(\mathcal{C} \otimes X, \mathcal{D}),\]
    where on the right we regard $X$ as an $\mathcal{S}$-\icat{}.
  \item The underlying groupoid object $\iota \FunV(\mathcal{C},
    \mathcal{D})$ of the Segal space $\FunV(\mathcal{C}, \mathcal{D})$
    is $\Map(\mathcal{C} \otimes E^{\bullet}, \mathcal{D})$. (Thus the
    ``correct'' space of objects of $\FunV(\mathcal{C}, \mathcal{D})$
    is the geometric realization $|\Map(\mathcal{C} \otimes E^{\bullet}, \mathcal{D})|$.)
  \item If $\mathcal{D}$ is a complete $\mathcal{V}$-\icat{}, then the
    Segal space $\FunV(\mathcal{C}, \mathcal{D})$ is
    complete.
  \end{enumerate}
\end{proposition}
\begin{proof}
  Tensoring $\mathcal{V}$-\icats{} with $\mathcal{S}$-\icats{}
  preserves colimits in each variable by \cite{enr}*{Corollary
    4.3.17}, so to prove (i) it suffices to show that the $\mathcal{S}$-\icats{}
  $[n]$ form a coSegal object, i.e. that the natural maps $[1]
  \amalg_{[0]} \cdots \amalg_{[0]} [1] \to [n]$ are
  equivalences in $\Algcat(\mathcal{S})$. 

  Recall from \cite{enr}*{\S 3.3} that there is a free-forgetful
  adjunction between $\mathcal{S}$-\icats{} and
  \emph{$\mathcal{S}$-graphs}, where an $\mathcal{S}$-graph with space
  of objects $X$ is just a
  functor $X \times X \to \mathcal{S}$. Let $\mathcal{G}_{n}$ denote the $\mathcal{S}$-graph
  with objects $\{0, \ldots, n\}$ and \[\mathcal{G}_{n}(i,j) =
  \begin{cases}
    *, & i < j \\
    \emptyset, & j \geq i.
  \end{cases}\] Then it is easy to see that $[n]$ is the free
  $\mathcal{S}$-\icat{} on the graph $\mathcal{G}_{n}$. Moreover, it
  is obvious that the map $\mathcal{G}_{1} \amalg_{\mathcal{G}_{0}}
  \cdots \amalg_{\mathcal{G}_{0}} \mathcal{G}_{1} \to \mathcal{G}_{n}$
  is an equivalence of $\mathcal{S}$-graphs. Since the formation of free
  $\mathcal{S}$-\icats{} preserves colimits, this implies that
  $[\bullet]$ is a coSegal object, which proves (i).

  Every Segal space can be canonically written as a colimit of a
  diagram of the objects $[n]$. Specifically, the Segal space
  $X$ is the coend of 
  \[ \bar{X} \colon \simp \times \simp^{\op} \to \SegI, \qquad ([n], [m])
  \mapsto \colim_{X_{m}} [n].\]
  Since $\Map([n], \FunV(\mathcal{C}, \mathcal{D})) \simeq
  \Map(\mathcal{C} \otimes [n], \mathcal{D})$ we then have
  \[ 
  \begin{split}
    \Map(X, \FunV(\mathcal{C}, \mathcal{D})) & \simeq \Map(\colim_{\txt{Tw}(\simp)}
  \bar{X}, \FunV(\mathcal{C}, \mathcal{D}))\\
  & \simeq \lim_{\txt{Tw}(\simp)}\Map(\bar{X},
  \FunV(\mathcal{C}, \mathcal{D})) \\
  & \simeq \lim_{\txt{Tw}(\simp)} \Map(\mathcal{C} \otimes
  \bar{X}, \mathcal{D}) \\
  & \simeq \Map(\mathcal{C} \otimes X,
  \mathcal{D}),
  \end{split}
  \]
  which proves (ii).

  The underlying groupoid object of a Segal space $X$ is
  $\Map(E^{\bullet}, X)$. By (ii), the underlying groupoid object of
  $\FunV(\mathcal{C},\mathcal{D})$ is therefore $\Map(\mathcal{C}
  \otimes E^{\bullet}, \mathcal{D})$, which proves (iii). The Segal
  space $\FunV(\mathcal{C}, \mathcal{D})$ is complete \IFF{} the
  canonical map from the colimit of this simplicial space to
  $\FunV(\mathcal{C},\mathcal{D})_{0}$ is an equivalence. By
  \cite{enr}*{Corollary 5.5.10} we know that if $\mathcal{D}$ is
  complete then \[|\iota\FunV(\mathcal{C}, \mathcal{D})| \simeq
  \Map(\mathcal{C}, \mathcal{D}) \simeq
  \FunV(\mathcal{C},\mathcal{D})_{0},\]
  thus $\FunV(\mathcal{C},\mathcal{D})$ is complete.
\end{proof}

Now suppose $\mathcal{V}$ is a presentably symmetric monoidal \icat{},
i.e. a symmetric monoidal \icat{} compatible with small colimits such
that the underlying \icat{} is presentable. Then by
\cite{enr}*{Corollary 4.3.16} and \cite{enr}*{Proposition 5.7.16} the
\icats{} $\AlgCatV$ and $\CatIV$ are also presentably symmetric
monoidal. This implies that $\AlgCatV$ and $\CatIV$ have internal Hom
objects; we write $\mathcal{D}^{\mathcal{C}}$ for the internal Hom for
maps $\mathcal{C} \to \mathcal{D}$ in $\AlgCatV$.

Let's check that the underlying \icat{} of the internal Hom
$\mathcal{D}^{\mathcal{C}}$ is precisely the functor \icat{}
$\FunV(\mathcal{C}, \mathcal{D})$:
\begin{proposition}
  Let $\mathcal{V}$ be a presentably symmetric monoidal \icat{}, and
  suppose $\mathcal{C}$ and $\mathcal{D}$ are $\mathcal{V}$-\icats{}.
  \begin{enumerate}[(i)]
  \item If $\mathcal{D}$ is complete, then the $\mathcal{V}$-\icat{}
    $\mathcal{D}^{\mathcal{C}}$ is complete for any $\mathcal{C}$. Moreover,
    $\mathcal{D}^{\mathcal{C}}$ is also the internal Hom in $\CatIV$.
  \item Write $t \colon
    \mathcal{S} \to \mathcal{V}$ for the unique colimit-preserving
    strong monoidal functor sending $*$ to the unit $I$; by
    \cite{enr}*{Proposition A.81} this has a lax monoidal right
    adjoint $u \colon \mathcal{V} \to \mathcal{S}$ given by $\Map(I,
    \blank)$. Then $\Map(E^{0}\otimes [\bullet], \mathcal{C})$ is the
    Segal space corresponding to the $\mathcal{S}$-\icat{}
    $u_{*}\mathcal{C}$.
  \item  The Segal space corresponding to the
    $\mathcal{S}$-\icat{} $u_{*}\mathcal{D}^{\mathcal{C}}$ underlying
    the internal Hom is 
    $\Fun^{\mathcal{V}}(\mathcal{C}, \mathcal{D})$.
  \end{enumerate}
\end{proposition}

\begin{proof}
  To prove (i), we must show that $\Map(E^{0},
  \mathcal{D}^{\mathcal{C}}) \to \Map(E^{1},
  \mathcal{D}^{\mathcal{C}})$ is an equivalence. Passing to left
  adjoints this is $\Map(\mathcal{C}, \mathcal{D}) \to \Map(E^{1}
  \otimes \mathcal{C}, \mathcal{D})$, which is an equivalence since
  $\mathcal{C} \otimes E^{1} \to \mathcal{C}$ is a local equivalence
  by \cite{enr}*{Proposition 4.45}.

  Since $\mathcal{D}^{\mathcal{C}}$ is complete we have, for any
  complete $\mathcal{V}$-\icat{} $\mathcal{A}$,
  \[
  \begin{split}
  \Map_{\CatIV}(\mathcal{A}, \mathcal{D}^{\mathcal{C}}) & \simeq
  \Map_{\AlgCatV}(\mathcal{A}, \mathcal{D}^{\mathcal{C}}) \simeq 
  \Map_{\AlgCatV}(\mathcal{A} \otimes \mathcal{C}, \mathcal{D}) \\ & \simeq 
  \Map_{\CatIV}(\mathcal{A} \otimes \mathcal{C}, \mathcal{D}),  
  \end{split}
  \]
  hence $\mathcal{D}^{\mathcal{C}}$ is also the internal hom in
  $\CatIV$.

  To prove (ii), observe that the Segal space corresponding to $u_{*}\mathcal{C}$
  is \[\Map_{\Algcat(\mathcal{S})}([n], u_{*}\mathcal{C}) \simeq
  \Map_{\Algcat(\mathcal{V})}(t_{*}[n], \mathcal{C}) \simeq
  \Map_{\Algcat(\mathcal{V})}(E^{0} \otimes [n], \mathcal{C}).\]

  Thus the Segal space associated to $u_{*}\mathcal{C}^{\mathcal{D}}$ is
  given by \[\Map_{\Algcat(\mathcal{V})}(E^{0}\otimes [\bullet], \mathcal{D}^{\mathcal{C}})
  \simeq \Map_{\Algcat(\mathcal{V})}(\mathcal{C} \otimes [\bullet],
  \mathcal{D}) \simeq \FunV(\mathcal{C}, \mathcal{D}).\qedhere\]
\end{proof}

\section{The $(\infty,2)$-Category of Enriched
  $\infty$-Categories}\label{sec:infty2}

The double \icat{} $\fALGcat(\mathcal{V})$ has two underlying
$(\infty,2)$-categories: one where the 1-morphisms are bimodules, and
one where they are functors; we write $\ALGcat(\mathcal{V})$ for the latter. Our
goal in this section is to show that the 2-morphisms in $\ALGcat(\mathcal{V})$ can
be identified with natural transformations, as we defined them in the
previous section. More precisely, we'll show:
\begin{proposition}\label{propn:FunVismaps}
  Let $\mathcal{C}$ and $\mathcal{D}$ be $\mathcal{V}$-\icats{}. There
  is a natural equivalence
  \[ \ALGcat(\mathcal{V})(\mathcal{C}, \mathcal{D}) \simeq
  \Fun_{\mathcal{V}}(\mathcal{C}, \mathcal{D}).\]
\end{proposition}

\begin{remark}
  The two underlying $(\infty,2)$-categories of a double \icat{}
  $\mathcal{X} \to \Dop$ can be described as follows: If $F \colon
  \Dop \to \CatI$ is the associated functor, then regarding $\CatI$ as
  the \icat{} of complete Segal spaces allows us to think of $F$ as a
  double Segal space $\Dop \times \Dop \to \mathcal{S}$. Now we can
  take the underlying 2-fold Segal space in either direction, which
  amounts to replacing $F_{nm}$ by its pullback along the degeneracy
  map from $F_{00}$ to $F_{n0}$ or $F_{m0}$ --- see
  \cite{spans}*{Proposition 2.13} for a precise statement.

\end{remark}

If $X$ is a space, let's abbreviate $\Dop_{X}[n] := \Dop_{X,\ldots,X}
\to \Dopn$. Let $\pi_{n} \colon [n] \to [0]$ be the unique map; this
induces maps $\pi_{n,*} \colon \Dop_{X}[n] \to \Dop_{X}$ and thus
$\pi_{n}^{*} \colon \fALGcat(\mathcal{V})_{0} \to
\fALGcat(\mathcal{V})_{n}$. If $\mathcal{C}$ and $\mathcal{D}$ are
$\mathcal{V}$-\icats{}, the \icat{} of maps from $\mathcal{C}$ to
$\mathcal{D}$ in $\ALGcat(\mathcal{V})$ can be identified with the Segal space
  \[ [n] \mapsto
  \Map_{\fALGcat(\mathcal{V})_{n}}(\pi_{n}^{*}\mathcal{C}, \pi_{n}^{*}\mathcal{D}).\]
Our goal is then to show that there is a natural equivalence between
the spaces $\Map_{\fALGcat(\mathcal{V})_{n}}(\pi_{n}^{*}\mathcal{C},
\pi_{n}^{*}\mathcal{D})$ and $\Map_{\Algcat(\mathcal{V})}(\mathcal{C}
\otimes [n], \mathcal{D})$. To do this we'll relate the
$\Dop_{X}[n]$-algebra $\pi_{n}^{*}\mathcal{C}$ to the $\Dop_{X \times
\{0,\ldots,n\}}$-algebra $\mathcal{C} \otimes [n]$.

Let $I_{n}\colon \Dopn \to \Dop_{\{0,\ldots,n\}}$ denote the obvious
inclusion, given on objects by sending $\phi \colon [m] \to [n]$ to
the list $(\phi(0),\ldots,\phi(m))$, and write $\xi_{n}$ for the
projection $\Dop_{\{0,\ldots,n\}} \to \Dop$. Then for any space $X$ we
have a commutative diagram
  \[
  \begin{tikzcd}
    \Dop_{X}[n] \arrow[bend left]{rr}{\pi_{n}} \arrow{d} \arrow{r}{I_{n}}& \Dop_{X \times \{0,\ldots, n\}} \arrow{d} \arrow{r}{\xi_{n}}& \Dop_{X} \arrow{d}\\
    \Dopn \arrow[bend right]{rr}{\pi_{n}}\arrow{r}{I_{n}}& \Dop_{\{0,\ldots,n\}} \arrow{r}{\xi_{n}}& \Dop,
  \end{tikzcd}
  \]
where both squares are Cartesian, and the maps are all morphisms of \gnsiopds{}.

\begin{lemma}\label{lem:DopXnext}
  For any space $X$, the inclusion $I_{n} \colon \DopX[n] \to \Dop_{X
    \times \{0,\ldots, n\}}$ is extendable.
\end{lemma}
\begin{proof}
  For $\xi \in \Dop_{X \times \{0,\ldots, n\}}$, the \icat{}
  $\DopX[n]^{\txt{act}}_{/\xi}$ is empty if $\xi \notin
  \DopX[n]$, or has a final object if $\xi \in \DopX[n]$. The
  extendability condition follows immediately from this.
\end{proof}

Suppose $\mathcal{V}$ is a monoidal \icat{} compatible with small
colimits, and let $\mathcal{C}$ be a $\mathcal{V}$-\icat{} with space
of objects $X$. Tensoring $\mathcal{C}$ with the map $[n] \to [0]$ we
get a map $\mathcal{C} \otimes [n] \to \mathcal{C}$ lying over the projection
$X \times \{0,\ldots,n\} \to X$. Thus this gives a map $\mathcal{C}
\otimes [n] \to \xi_{n}^{*}\mathcal{C}$ of $\Dop_{X \times
  \{0,\ldots,n\}}$-algebras.
\begin{lemma}
  In the situation above, the induced map \[I_{n}^{*}(\mathcal{C}
  \otimes [n]) \to I_{n}^{*}\xi_{n}^{*} \simeq
  \pi_{n}^{*}\mathcal{C}\]
  is an equivalence.
\end{lemma}
\begin{proof}
  It suffices to observe that for $i \leq j$ and any $x, y \in X$
  the morphism \[(\mathcal{C}\otimes [n])((x,i),(y,j)) \to
  \mathcal{C}(x,y)\] is an equivalence.
\end{proof}

This map therefore has an inverse $\pi_{n}^{*}\mathcal{C} \to
I_{n}^{*}(\mathcal{C} \otimes [n])$, and since $I_{n}^{*}$ has a left
adjoint by Lemma~\ref{lem:DopXnext} there is a natural map
$I_{n,!}\pi_{n}^{*}\mathcal{C} \to \mathcal{C} \otimes [n]$ of
$\Dop_{X\times\{0,\ldots,n\}}$-algebras. 
\begin{proposition}\label{propn:CtensornIn}
  In the situation above, the morphism
  $I_{n,!}\pi_{n}^{*}\mathcal{C} \to \mathcal{C} \otimes [n]$ is an
  equivalence.
\end{proposition}
\begin{proof}
  Again observe that for $\xi \in \Dop_{X \times \{0,\ldots, n\}}$, the \icat{}
  $\DopX[n]^{\txt{act}}_{/\xi}$ is empty if $\xi \notin
  \DopX[n]$, or has a final object if $\xi \in \DopX[n]$. By the
  definition of left operadic Kan extensions we therefore see that for
  $x,y \in X$ and $i,j \in \{0,\ldots,n\}$ we have
  \[ I_{n,!}\pi_{n}^{*}\mathcal{C}((x,i), (y,j)) \simeq
  \begin{cases}
    \emptyset, & i > j \\
    \mathcal{C}(x,y), & i \leq j.
  \end{cases}
  \]
  The forgetful functor from $\Dop_{X \times \{0,\ldots,
    n\}}$-algebras to functors $(X \times \{0,\ldots,n\})^{\times
    2}\to \mathcal{V}$ is conservative by \cite{enr}*{Lemma A.5.5}, so
  this completes the proof.
\end{proof}

Now consider the algebra fibration $\Alg(\mathcal{V}) \to
\OpdInsg$. Since $\fALGcat(\mathcal{V})_{n}$ is pulled back from this,
if $\mathcal{C}$ and $\mathcal{D}$ are $\mathcal{V}$-\icats{} with
spaces of objects $X$ and $Y$, respectively, then we have a pullback square
\nolabelcsquare{\Map_{\fALGcat(\mathcal{V})_{n}}(\pi_{n}^{*}\mathcal{C},
  \pi_{n}^{*}\mathcal{D})}{\Map_{\Alg(\mathcal{V})}(\pi_{n}^{*}\mathcal{C},
  \pi_{n}^{*}\mathcal{D})}{\Map_{\mathcal{S}}(X, Y)^{\times
    (n+1)}}{\Map_{\OpdInsg}(\DopX[n], \Dop_{Y}[n]).}
By Proposition~\ref{propn:CtensornIn} we also have a natural
equivalence \[\Map_{\Alg(\mathcal{V})}(\pi_{n}^{*}\mathcal{C},
\pi_{n}^{*}\mathcal{D}) \simeq
\Map_{\Alg(\mathcal{V})}(I_{n,!}\pi_{n}^{*}\mathcal{C},
\xi_{n}^{*}\mathcal{D}).\]
Moreover, if we consider the diagram
\[
\begin{tikzcd}
  \Map'(\mathcal{C} \otimes [n], \xi_{n}^{*}\mathcal{D}) 
  \arrow{d} \arrow{r}& \Map(\mathcal{C} \otimes [n], \xi_{n}^{*}\mathcal{D})
  \arrow{d} \arrow{r}& \Map(\mathcal{C} \otimes [n], \mathcal{D}) \arrow{d}\\
  \Map(X^{\amalg (n+1)}, Y) \arrow{r}& \Map(X^{\amalg (n+1)}, Y^{\amalg (n+1)}) \arrow{r}& \Map(X^{\amalg (n+1)}, Y),
\end{tikzcd}
\]
where the left square is defined to be a pullback square, then the
top composite map is an equivalence: since the bottom composite map
is the identity, it suffices to check that the map is an equivalence
on each fibre, which is clear. Thus we have identified both 
$\Map_{\Alg(\mathcal{V})}(\pi_{n}^{*}\mathcal{C},
\pi_{n}^{*}\mathcal{D})$ and $\Map(\mathcal{C} \otimes [n],
\mathcal{D})$ with the same pullback, which completes the proof of
Proposition~\ref{propn:FunVismaps}.

\begin{corollary}
  Let $\CATIV$ be the underlying 2-fold Segal space of
  $\fCAT(\mathcal{V})$ with functors as 1-morphisms. Then this 2-fold
  Segal space is complete.
\end{corollary}
\begin{proof}
  The underlying Segal space
  of $\CATIV$ is that associated to the \icat{} $\CatIV$, and so is
  complete. Using the completeness criterion of
  \cite{spans}*{Theorem 5.18} it then suffices to show that the Segal
  space $\CATIV(\mathcal{C}, \mathcal{D})$ of maps from $\mathcal{C}$
  to $\mathcal{D}$ is complete for all complete $\mathcal{V}$-\icats{}
  $\mathcal{C}$ and $\mathcal{D}$, which follows from combining
  Proposition~\ref{propn:FunVismaps} and Proposition~\ref{propn:[n]coSegal}.
\end{proof}

\section{Functoriality and Monoidal Structures}\label{sec:func}
In this section we consider the functoriality in $\mathcal{V}$ of the
double \icat{} of $\mathcal{V}$-\icats{}. Here we restrict ourselves
to the ``algebraic'' or pre-localized case of the double \icats{}
$\fALGcat(\mathcal{V})$ --- since composition with a
colimit-preserving monoidal functor does not usually preserve complete
objects (cf. \cite{enr}*{\S 5.7}), to establish functoriality for the
double \icats{} $\fCAT(\mathcal{V})$ we must first show that the
\icat{} of $\mathcal{C}$-$\mathcal{D}$-bimodules in $\mathcal{V}$ is
invariant under fully faithful and essentially surjective functors of
$\mathcal{C}$ and $\mathcal{D}$. This result is most naturally proved
as a consequence of the Yoneda Lemma (in the form of the
representability of the \icat{} of bimodules), and so we postpone it
to a sequel to this paper.

\begin{definition}
  In the previous section we constructed a functor $
  \simp_{\mathcal{S}} \to \OpdInsg$ that sends $(X_{0},\ldots,X_{n})$ to
  $\DopXn$. Combining this with the algebra functor \[\Alg \colon (\OpdInsg)^{\op}
  \times \OpdInsg \to \CatI,\] we get a functor
  $(\simp_{\mathcal{S}})^{\op} \times \OpdInsg \to \CatI$ that sends
  the object $((X_{0},\ldots,X_{n}), \mathcal{O})$ to
  $\Alg_{\DopXn}(\mathcal{O})$. Let $\ofALGcat^{\vee} \to \simp_{\mathcal{S}}
  \times (\OpdInsg)^{\op}$ be a Cartesian fibration associated to this
  functor, and then take $\ofALGcat \to \simp^{\op} \times \OpdInsg$ to
  be a coCartesian fibration associated to the composite \[\ofALGcat^{\vee} \to
  \simp_{\mathcal{S}}
  \times (\OpdInsg)^{\op} \to \simp \times (\OpdInsg)^{\op}.\]
\end{definition}

\begin{remark}
  The coCartesian fibration $\ofALGcat \to \simp^{\op} \times \OpdInsg$
  determines a functor $\Dop \times \OpdInsg \to \CatI$ or
  $\OpdInsg \to \Fun(\Dop, \CatI)$.
\end{remark}

\begin{definition}
  Let $\LCatI^{\txt{coC}}$ denote the \icat{} of (large) \icats{} with
  small colimits and colimit-preserving functors. This \icat{} has a tensor product, constructed in
  \cite{HA}*{\S 4.8.1}, such that a map $\mathcal{C} \otimes
  \mathcal{D} \to \mathcal{E}$ is equivalent to a map $\mathcal{C}
  \times \mathcal{D} \to \mathcal{E}$ that preserves colimits
  separately in each variable. Then the \icat{} $\txt{Mon}_{\infty}^{\txt{coC}} :=
  \Alg_{\mathbb{E}_{1}}^{\Sigma}(\LCatI^{\txt{coC}})$ of associative
  algebras with respect to this tensor product is the \icat{}
  of monoidal \icats{} compatible with small colimits and
  colimit-preserving monoidal functors.
\end{definition}

\begin{definition}
  There is a forgetful functor $\txt{Mon}_{\infty}^{\txt{coC}} \to
  \LOpdInsg$. Let $\ofALGcat' \to \simp^{\op} \times
  \txt{Mon}_{\infty}^{\txt{coC}}$ be defined by the pullback along
  this of the obvious variant of the coCartesian fibration $\ofALGcat$
  where we allow the targets to be large. Then we define $\fALGcat$ to be the full
  subcategory of $\ofALGcat'$ spanned by the objects of
  $\fALGcat(\mathcal{V})$ for all $\mathcal{V}$ in
  $\txt{Mon}_{\infty}^{\txt{coC}}$.
\end{definition}

\begin{proposition}
  The restricted projection $\fALGcat \to \simp^{\op} \times
  \txt{Mon}_{\infty}^{\txt{coC}}$ is a coCartesian fibration.
\end{proposition}
\begin{proof}
  It suffices to prove that if $f \colon \mathcal{V}^{\otimes} \to
  \mathcal{W}^{\otimes}$
  is a colimit-preserving monoidal functor then for every composite
  algebra $M \colon \DopXn \to
  \mathcal{V}^{\otimes}$,  the composite map
  $f_{*}M \colon \DopXn \to \mathcal{V}^{\otimes} \to
  \mathcal{W}^{\otimes}$
  is also a composite algebra. In other words, we must show that the diagram
  \csquare{\Alg_{\LopXn}(\mathcal{V})}{\Alg_{\LopXn}(\mathcal{W})}{\Alg_{\DopXn}(\mathcal{V})}{\Alg_{\DopXn}(\mathcal{W})}{f_*}{\tau_{X_{0},\ldots,X_{n},!}}{\tau_{X_{0},\ldots,X_{n},!}}{f_*}
  commutes. This is a special case of \cite{enr}*{Lemma A.4.7}.
\end{proof}

The coCartesian fibration $\fALGcat \to \simp^{\op} \times
\txt{Mon}_{\infty}^{\txt{coC}}$ determines a functor
$\txt{Mon}_{\infty}^{\txt{coC}} \times \Dop \to \CatI$, or
equivalently $\txt{Mon}_{\infty}^{\txt{coC}} \to \Fun(\Dop,
\CatI)$. This functor factors through the full subcategory
$\Cat(\CatI)$ spanned by the double \icats{}, giving:
\begin{corollary}
  The double \icat{} $\fALGcat(\mathcal{V})$ is functorial in
  colimit-preserving monoidal functors between monoidal \icats{}
  $\mathcal{V}$ compatible with small colimits, i.e. it is given by a
  functor $\txt{Mon}_{\infty}^{\txt{coC}} \to \Cat(\CatI)$.
\end{corollary}

Next we want to show that the functors $\ofALGcat(\blank)$ and
$\fALGcat(\blank)$ are lax monoidal. This requires a somewhat more
involved construction, as we want the lax monoidal structure to be
defined in terms of the ``external product'' $M \boxtimes N$ of $M
\colon \DopXn \to \mathcal{V}^{\otimes}$ and $N \colon
\Dop_{Y_{0},\ldots,Y_{n}} \to \mathcal{W}^{\otimes}$ given by
\[ 
\begin{split}
M \boxtimes N \colon \Dop_{X_{0} \times Y_{0},\ldots, X_{n} \times
  Y_{n}} & \simeq \DopXn \times_{\Dopn} \Dop_{Y_{0},\ldots,Y_{n}} \\ &
\to
\DopXn \times_{\Dop} \Dop_{Y_{0},\ldots,Y_{n}} \\ & \to
\mathcal{V}^{\otimes} \times_{\Dop} \mathcal{W}^{\otimes},
\end{split}\] which
means that we must consider for each $n$ the fibre product of
\gnsiopds{} over $\Dopn$ in a compatible manner.

\begin{lemma}\label{lem:cocartmonfibred}
  Suppose $\mathcal{C}$ is an \icat{} with finite colimits. Then the
  functor $\mathcal{C} \to \CatI$ sending $x$ to $\mathcal{C}_{x/}$
  lifts to a functor from $\mathcal{C}$ to symmetric monoidal \icats{}
  sending $x$ to $(\mathcal{C}_{x/})^{\amalg_{x}}$. Let
  $(\mathcal{C}_{\bullet/})^{\amalg} \to \mathcal{C} \times
  \Gop$ be the coCartesian fibration of \iopds{} induced by
  this; then the forgetful functors $\mathcal{C}_{x/} \to \mathcal{C}$
  give rise to a morphism of \iopds{} $(\mathcal{C}_{\bullet/})^{\amalg} \to
  \mathcal{C}^{\amalg}$.
\end{lemma}
\begin{proof}
  Immediate from \cite{HA}*{Corollary 2.4.3.11} and
  \cite{HA}*{Proposition 2.4.3.16}.
\end{proof}

\begin{construction}\label{constr:1}
  We can apply Lemma~\ref{lem:cocartmonfibred} to $(\OpdInsg)^{\op}$
  to get a morphism of \iopds{}
  $((\OpdInsg)_{/\bullet}^{\op})^{\amalg} \to
  ((\OpdInsg)^{\op})^{\amalg}$. Combined with the lax monoidal algebra
  functor
\[ ((\OpdInsg)^{\op})^{\amalg} \times_{\Gop}
(\OpdInsg)^{\times} \to \CatI^{\times} \]
of \cite{nmorita}*{Definition 4.44}, this
gives a map of \iopds{}
\[ ((\OpdInsg)_{/\bullet}^{\op})^{\amalg} \times_{\Gop}
(\OpdInsg)^{\times} \to ((\OpdInsg)^{\op})^{\amalg} \times_{\Gop}
(\OpdInsg)^{\times} \to \CatI^{\times}, \]
with lax monoidal structure given, for $\mathcal{P},\mathcal{Q}$
\gnsiopds{} over $\mathcal{O}$, by 
\[ \Alg_{\mathcal{P}}(\mathcal{V}) \times
\Alg_{\mathcal{Q}}(\mathcal{V}) \to \Alg_{\mathcal{P} \times_{\Dop}
  \mathcal{Q}}(\mathcal{V}) \to \Alg_{\mathcal{P} \times_{\mathcal{O}}
  \mathcal{Q}}(\mathcal{V}).\]
\end{construction}

\begin{construction}\label{constr:2}
  The functors $\Dop_{\blank,\ldots,\blank} \colon \mathcal{S}^{\times
    (n+1)} \to (\OpdInsg)_{/\Dopn}$ preserve products, and considering
  their naturality in $[n] \in \simp$ this gives rise to a natural
  transformation $\Dop \times \Delta^{1} \to \CatI^{\txt{prod}}$ of
  functors from $\Dop$ to the \icat{} $\CatI^{\txt{prod}}$ of \icats{}
  with products (and product-preserving functors). Here the functor
  $\phi^{* }\colon (\OpdInsg)_{/\Dopn} \to (\OpdInsg)_{\Dop_{/[m]}}$
  induced by a map $\phi \colon [m] \to [n]$ in $\simp$ is given by
  taking pullbacks along $\phi_{*} \colon \Dop_{/[m]} \to \Dopn$.

  By \cite{HA}*{Corollary 2.4.1.9}, the \icat{} $\CatI^{\txt{prod}}$
  is equivalent to the full subcategory of the \icat{}
  $\Alg_{\Gop}^{\Sigma}(\CatI)$ of symmetric monoidal
  \icats{} (and symmetric monoidal functors) spanned by the Cartesian
  symmetric monoidal \icats{}. Thus we get a map \[\Dop \times
  \Delta^{1} \to \CatI^{\txt{prod}} \to
  \Alg_{\Gop}^{\Sigma}(\CatI).\]

  Since $(\blank)^{\op} \colon \CatI \to \CatI$ preserves products, it
  induces a functor \[(\blank)^{\op} \colon
  \Alg_{\Gop}^{\Sigma}(\CatI) \to
  \Alg_{\Gop}^{\Sigma}(\CatI).\] Combining this with the
  forgetful functor from symmetric monoidal \icats{} to coCartesian
  fibrations over $\Gop$, we get a functor
  \[ \Dop \times
  \Delta^{1} \to 
  \Alg_{\Gop}^{\Sigma}(\CatI)
  \xto{(\blank)^{\op}}   \Alg_{\Gop}^{\Sigma}(\CatI)
  \to \Cat_{\infty/\Gop}^{\txt{coCart}} \simeq
  \Fun(\Gop, \CatI). \]
  Equivalently, this is a functor $\Delta^{1} \to \Fun(\Dop \times
  \Gop, \CatI)$, or a morphism
  \nolabelopctriangle{\simp_{\mathcal{S}^{\op}}^{\op,\times}}{\Phi^{*}((\OpdInsg)_{/\bullet}^{\op})^{\amalg}}{\Dop
    \times \Gop,}
   of coCartesian fibrations over $\Dop \times \Gop$, where \[\Phi^{*}((\OpdInsg)_{/\bullet}^{\op})^{\amalg} \to \Dop
    \times \Gop\] denotes the pullback of
    $((\OpdInsg)_{/\bullet}^{\op})^{\amalg} \to (\OpdInsg)^{\op}
    \times \Gop$ along \[\Phi := \Dop_{(\blank)} \times \id \colon \Dop
    \times \Gop \to (\OpdInsg)^{\op}
    \times \Gop.\]
  Composing with the projection $\Dop \times
  \Gop \to \Gop$, we get in particular a morphism
  of generalized \iopds{}
  \[ \simp_{\mathcal{S}^{\op}}^{\op,\times} \to
  \Phi^{*}((\OpdInsg)_{/\bullet}^{\op})^{\amalg} \to
  ((\OpdInsg)_{/\bullet}^{\op})^{\amalg}.\] This sends an object
  $((X^{1}_{0},\ldots,X_{n}^{1}),\ldots,(X^{m}_{0},\ldots,X^{m}_{n})$
  in the fibre over $([n], \langle m \rangle)$ to
  $(\Dop_{X_{0}^{1},\ldots,X_{n}^{1}},\ldots,\Dop_{X_{0}^{m},\ldots,X_{n}^{m}})$.
\end{construction}

\begin{construction}\label{constr:3}
  Combining Constructions~\ref{constr:1} and \ref{constr:2}, we get a
  morphism of generalized \iopds{}
  \[\simp_{\mathcal{S}^{\op}}^{\op,\times} \times_{\Gop}
  (\OpdInsg)^{\times} \to 
((\OpdInsg)_{/\bullet}^{\op})^{\times} \times_{\Gop}
(\OpdInsg)^{\times}  \to \CatI^{\times}.\]
 By (a slight variant of) \cite{HA}*{Proposition 2.4.2.5} this corresponds to a
  monoid object $\simp_{\mathcal{S}^{\op}}^{\op,\times} \times_{\Gop}
  (\OpdInsg)^{\times} \to \CatI$. Let \[(\ofALGcat^{\otimes})^{\vee} \to (\simp_{\mathcal{S}^{\op}}^{\op,\times} \times_{\Gop}
  (\OpdInsg)^{\times})^{\op}\] be a Cartesian fibration associated to
  this functor. Then the composite
  \[ 
  \begin{split}
(\ofALGcat^{\otimes})^{\vee} & \to
  (\simp_{\mathcal{S}^{\op}}^{\op,\times} \times_{\Gop}
  (\OpdInsg)^{\times})^{\op} \\ & \to ((\Dop \times \Gop)
  \times_{\Gop} (\OpdInsg)^{\times})^{\op} \\ & \simeq (\Dop
  \times (\OpdInsg)^{\times})^{\op}
  \end{split}
\] is again a Cartesian fibration,
  and we define $\ofALGcat^{\otimes} \to \Dop \times
  (\OpdInsg)^{\times}$ to be a coCartesian fibration associated to
  this Cartesian fibration. This is a coCartesian fibration of \gnsiopds{}.
\end{construction}

\begin{construction}
  The symmetric monoidal structure on $\LCatI^{\txt{coC}}$ induces a
  tensor product on $\txt{Mon}_{\infty}^{\txt{coC}}$, and the
  forgetful functor from $\txt{Mon}_{\infty}^{\txt{coC}}$ to
  $\LOpdInsg$ can be enhanced to a lax monoidal functor
  $\txt{Mon}_{\infty}^{\txt{coC},\otimes} \to (\LOpdInsg)^{\times}$.
  Let $(\ofALGcat^{\otimes})' \to \simp^{\op} \times
  \txt{Mon}_{\infty}^{\txt{coC},\otimes}$ be defined by the pullback
  \nolabelcsquare{(\ofALGcat^{\otimes})'}{\ofALGcat^{\otimes}}{\simp^{\op}
    \times \txt{Mon}_{\infty}^{\txt{coC},\otimes}}{\Dop \times
    (\LOpdInsg)^{\times}} along this, where the right vertical map is
  the variant of the map of Construction~\ref{constr:3} where we allow
  the target \gnsiopds{} to be large. This gives a coCartesian
  fibration of \gnsiopds{}.
\end{construction}

\begin{definition}
  We define $\fALGcat^{\otimes}$ to be the full subcategory of
  $(\ofALGcat^{\otimes})'$ spanned by those objects that correspond to
  lists of objects of $\fALGcat$.
\end{definition}

\begin{proposition}\label{propn:extprodcomp}
  The restricted projection $\fALGcat^{\otimes} \to \simp^{\op} \times
  \txt{Mon}_{\infty}^{\txt{coC},\otimes}$ is a coCartesian fibration.
\end{proposition}

To see this we use the following technical observation:
\begin{lemma}\label{lem:LopXYcof}
  Let $X_{0},\ldots,X_{n}$ and $Y_{0},\ldots,Y_{n}$ be spaces, and
  suppose $(\Xi, \mathrm{H})$ is an object of $\DopXYn$ over $\xi \in
  \Dopn$. Then the map
  \[ (\LopXYn)^{\txt{act}}_{/(\Xi, \mathrm{H})} \to (\LopXn)^{\txt{act}}_{/\Xi}
  \times (\LopXn)^{\txt{act}}_{/\mathrm{H}} \]
  is cofinal.
\end{lemma}
\begin{proof}
  We have a pullback square
  \nolabelcsquare{(\LopXYn)^{\txt{act}}_{/(\Xi,
      \mathrm{H})}}{(\LopXn)^{\txt{act}}_{/\Xi} \times
    (\LopXn)^{\txt{act}}_{/\mathrm{H}}}{(\Lopn)^{\txt{act}}_{/\xi}}{(\Lopn)^{\txt{act}}_{/\xi}
    \times (\Lopn)^{\txt{act}}_{/\xi}} where the vertical maps are
  left fibrations by Lemma~\ref{lem:LopXnactXifib}, and the bottom
  horizontal map is cofinal since the \icat{}
  $(\Lopn)^{\txt{act}}_{/\xi}$ is sifted by \cite{nmorita}*{Lemma
    5.7}. It therefore follows from \cite{HTT}*{Proposition 4.1.2.15}
  and \cite{HTT}*{Remark 4.1.2.10} that the top horizontal map is also
  cofinal.
\end{proof}

\begin{proof}[Proof of Proposition~\ref{propn:extprodcomp}]
  It suffices to show that given a pair of composite algebras $M \colon \DopXn
  \to \mathcal{V}^{\otimes}$ and $N \colon \Dop_{Y_{0},\ldots,Y_{n}}
  \to \mathcal{W}^{\otimes}$, their external product
  \[ M \boxtimes N \colon \Dop_{X_{0} \times Y_{0},\ldots, X_{n}
    \times Y_{n}} \to \mathcal{V}^{\otimes} \times_{\Dop}
  \mathcal{W}^{\otimes}\]
  is also composite. This follows from 
  Lemma~\ref{lem:LopXYcof} together with the definition of left
  operadic Kan extensions.
\end{proof}

\begin{corollary}\label{cor:ALGcatlaxmon}
  $\fALGcat^{\otimes}$ determines a lax monoidal functor
  $\txt{Mon}_{\infty}^{\txt{coC},\otimes} \to \Cat(\CatI)^{\times}$.
\end{corollary}

\begin{corollary}
  Suppose $\mathcal{O}$ is a (symmetric) \iopd{} and $\mathcal{V}$ is
  an $\mathcal{O} \otimes \mathbb{E}_{1}$-monoidal \icat{} compatible
  with small colimits. Then the double \icat{}
  $\fALGcat(\mathcal{V})$ inherits a natural
  $\mathcal{O}$-monoidal structure.
\end{corollary}
\begin{proof}
  By the universal property of the Boardman-Vogt tensor product of
  \iopds{}, we can identify an $\mathcal{O}\otimes
  \mathbb{E}_{1}$-monoidal \icat{} compatible with small colimits with
  an $\mathcal{O}$-algebra in $\txt{Mon}_{\infty}^{\txt{coC}}$. By
  Corollary~\ref{cor:ALGcatlaxmon}, this implies that the functor
  $\fALGcat$ takes $\mathcal{V}$ to an $\mathcal{O}$-algebra in double
  \icats{}, i.e. an $\mathcal{O}$-monoidal double \icat{}.
\end{proof}

Lurie's additivity theorem, \cite{HA}*{Theorem 5.1.2.2}, gives an
equivalence \[\mathbb{E}_{n}\otimes \mathbb{E}_{1} \isoto \mathbb{E}_{n+1},\] so as a
special case we have:
\begin{corollary}
  Suppose $\mathcal{V}$ is an $\mathbb{E}_{n+1}$-monoidal \icat{}
  compatible with small colimits. Then the double \icat{}
  $\fALGcat(\mathcal{V})$ inherits a natural $\mathbb{E}_{n}$-monoidal
  structure.
\end{corollary}
Taking the colimit of these equivalences as $n \to \infty$, we get an
equivalence \[\mathbb{E}_{\infty}\otimes \mathbb{E}_{1} \isoto
\mathbb{E}_{\infty},\] which similarly gives:
\begin{corollary}
  Suppose $\mathcal{V}$ is a symmetric monoidal \icat{} compatible
  with small colimits. Then the double \icat{} $\fALGcat(\mathcal{V})$
  inherits a natural symmetric monoidal structure.
\end{corollary}

\end{document}